\documentclass[graybox]{article}

\usepackage{graphicx}        

\usepackage{amsmath}
\usepackage{amsfonts}
\usepackage{amssymb}
\usepackage{graphicx}
\usepackage{enumerate}
\usepackage{bbm}
\usepackage{bm}

\usepackage{graphicx}
\usepackage{bbm}
\usepackage{bm}
\usepackage{graphics}
\usepackage{color}
\usepackage{cite}

\newtheorem{theorem}{Theorem}
\newtheorem{lemma}[theorem]{Lemma}
\newtheorem{definition}[theorem]{Definition}
\newtheorem{conjecture}[theorem]{Conjecture}
\newtheorem{proposition}[theorem]{Proposition}
\newenvironment{proof}[1][Proof]{\textbf{#1.} }{\ \rule{0.5em}{0.5em}}

\begin{document}

\title{Towards direct $L^2$-bounds for maximal partial sums of Walsh--Fourier series: The case of dyadic partial sums}
\author{Joseph D. Lakey\footnote{New Mexico State University, {\tt jlakey@nmsu.edu}}}
%
%
\maketitle

\abstract{
We outline an approach to obtain direct $L^2$ estimates not requiring interpolation for so-called linearized partial sums operators associated with expansions in Walsh functions. We focus specifically on a simpler case of dyadic partial sums but also outline a second approach  to proving bounds on general linearized partial sums.
}


\section{Introduction and background}


\subsection{Convergence of Fourier series}
The study of pointwise convergence of Fourier series eventually culminated in the much-celebrated Carleson--Hunt theorem
\cite{carleson_1966,hunt_1968}, cf., \cite{fefferman_ae_1973,mozzochi_1971,meljbro_1982} and \cite{lacey_2004,arias_2002} for later perspectives.
The result led to further investigations addressing convergence of Fourier series and integrals in higher dimensional Euclidean settings, e.g., \cite{carbery_1988,Antonov_1996,antonov_multiple,mastylo_2018} and more recently \cite{lacey_thiele_2000,muscalu_2006};
convergence in topological group settings, e.g.,
\cite{hunt_taibleson_1971,stanton_tomas_1976,colzani_1990}  and more recently \cite{persson_2022,oniani_2020,areshidze_2025,gat_2009};
spaces on which maximal partial sum operators are bounded, e.g., \cite{lie_2017,souza_1984} and convergence of expansions in other orthogonal functions, e.g.,
\cite{moricz_1994,guadalupe_1992,badkov_1974,kita_1985}
including orthogonal polynomials or special functions, e.g., \cite{chen_1994,ciaurri_1999}; and other related questions.

The concept of lacunary  partial sums\footnote{Here lacunarity refers to indices of a (sub) sequence $S_{n_k}$ of partial sums---not to the terms themselves} has provided corresponding convergence results in each of these settings, including
\cite{oniani_2020,areshidze_2025,antonov_lacunary,plinio_2015,bailey_2014} and others.
In 1924  Kolmogorov \cite{Kolmogoroff1924} established almost-everywhere convergence of lacunary partial sums of Fourier series in $L^2[0,1]$ by means of comparison with
 Ces\`aro means. That approach did not consider any norm bound on a maximal partial sum operator.  
 
 In this work we consider dyadic partial sums $S_{2^N}$ of expansions in Walsh functions.  Almost-everywhere convergence of Walsh--Fourier series of functions in $L^2[0,1]$ was established by Billard \cite{billard19661967} shortly after Carleson's work appeared (cf., Hunt \cite{hunt_walsh_1971}). It can be regarded as the first extension of Carleson's methods to a different setting.   The fact that Walsh functions are characters of the Cantor group allowed Gosselin \cite{gosselin1973} to extend Carleson's approach to Vilenkin groups that have a parallel structure. Subsequent work by Gosselin \cite{gosselin1979} extended C.~Fefferman's approach  to almost everywhere convergence of Fourier series in \cite{fefferman_ae_1973} to the Walsh setting.  The script was flipped starting with Thiele's work in 2000 \cite{thiele2000} when the Walsh setting was seen to provide a somewhat cleaner context for a phase space (wavepacket) approach to questions of boundedness of a family of operators that included partial sum operators and certain multilinear singular integral operators, cf., Muscalu et al.~\cite{muscalu_etal_walsh_2004}.  The approach  to  uniform boundedness of linearized partial sum operators outlined here is fundamentally different from
 prior approaches and uses special properties of Walsh functions.

Our goal here is in one sense much more modest than the aforementioned works: we seek concrete bounds on operator norms of a restricted family of \emph{dyadic} partial sum operators in the Walsh setting. On the other hand the ultimate goal (not achieved here) is a sharp uniform $L^2\to L^2$ bound on this family of operators, which is somewhat new in the study of Fourier series where maximal partial sum bounds are initially established on a different space ($L(\log_+L)^{1+\delta}$ in Carleson's work) with $L^2$ bounds following by interpolation. The program at hand seeks a direct $L^2$ bound on dyadic partial sums, using
 optimization methods to identify families of matrices corresponding to  dyadic partial sums that should have maximal norms.  
We outline very briefly at the end a different \emph{dilation} approach that should also extend to provide
direct $L^2$ estimates for general (not necessarily dyadic) maximal partial sums of Walsh expansions.

In the case of dyadic partial sums of Walsh--Fourier series we conjecture an explicit optimal bound on a class of linearized
dyadic partial sum operators, provide arguments for the plausibility of the conjecture and support these arguments with
concrete numerical evidence.  The partial sum estimates are applied to  dyadic step functions that are constant
on dyadic intervals $\left[k/2^N,(k+1)/{2^N}\right)$, $k=0,\dots, 2^N-1$.
 This allows rephrasing 
boundedness of partial sum operators in terms of boundedness of a family of what we refer to as (dyadically) truncated Walsh--Hadamard (DTWH) matrices.

Here is an outline of the presentation. In Sect.~\ref{sect:walsh} we review the Walsh functions and show that on dyadic step functions on $[0,1]$, certain linearized Walsh--Fourier (dyadic) partial sum operators can be represented in terms of 
DTWH matrices.  We then state a conjecture (Conj.~\ref{thm:uniform_bound}) that would provide a sharp bound on norms of maximal dyadic partial sum operators as well as a second (looser) conjecture  (Conj.~\ref{conj:uniform_bound_general}) on uniform norm bounds for  linearized Walsh--Fourier partial sum operators that are not-necessarily dyadic.  In Sect.~\ref{sect:properties} we formalize a secondary conjecture (Conj.~\ref{claim:one_node})   based on \emph{branching} properties of the sets of columns of  DTWH matrices that would lead to a proof of Conj.~\ref{thm:uniform_bound}.  In Sect.~\ref{sect:evidence} we state a specific case (Conj.~\ref{claim:bnk_opt}) of Conj.~\ref{claim:one_node}, outline a reduction to this special case,  and provide evidence of its validity.  In Sect.~\ref{sect:rigorous} we outline a path to fill in further details needed to prove Conj.~\ref{claim:bnk_opt} from which the main conjecture, Conj.~\ref{thm:uniform_bound}, would follow. Finally, in Sect.~\ref{sect:discussion} we discuss general truncations 
and outline very broadly steps that are needed to prove Conj.~\ref{conj:uniform_bound_general}. A proposition supporting 
Conj.~\ref{claim:bnk_opt} that is stated in Sect.~\ref{sect:evidence}
is proved in Appendix.~\ref{appendix:proof_Fcritical}.

When referring to the norm of a matrix, we mean its operator norm (largest singular value).  For specific symmetric matrices whose largest (in absolute value) eigenvalues happen to be positive, we will refer to corresponding eigenvectors as \emph{norm eigenvectors} (and the corresponding eigenvalue as the norm eigenvalue, or simply the norm).

\section{Partial sums of  Walsh--Fourier series \label{sect:walsh}}

\subsection{Walsh functions}

\bigskip
 In the \emph{Paley ordering} used here (cf., \cite{thiele2000}), the $n$th Walsh function $W_n(t)$, $n\in\mathbb{N}$, is defined recursively by
\begin{eqnarray*}
W_0(t)&=&1,\quad 0\leq t<1\\
W_{2n}(t)&=& \begin{cases} W_n(2t),&0\leq t<1/2\\
 W_n(2t-1),& \quad 1/2\leq t<1 \end{cases} \\
W_{2n+1}(t)&=& \begin{cases}  W_n(2t), & \quad 0\leq t<1/2\\
 -W_n(2t-1),&\quad 1/2\leq t<1\ \end{cases} 
\end{eqnarray*}
The Walsh functions can also be expressed in terms of Rademacher functions $r_k(t)={\rm sign}\, \sin(2^k\pi t)$ by $W_n(t)=\prod_{k=0}^\infty r_k(t)^{n_k}=(-1)^{\sum_{k=0}^\infty n_k t_{k+1}}$
where $n=\sum n_k 2^k$ ($n_k\in\{0,1\}$) is the binary decomposition of $n$ and $t=\sum_{k=1}^\infty t_k 2^{-k}$ ($t_k\in\{0,1\}$) is the dyadic decomposition of $t\in [0,1)$ (terminating in zeros if $t$ is a dyadic rational).
The Walsh functions form an orthonormal basis for $L^2[0,1]$.

\subsection{Walsh--Hadamard matrices and finite-dimensional representations of dyadic step functions}

As a basis for dyadic step functions up to level $N$, that is, functions constant on intervals $\left[\frac{k}{2^N},\frac{k+1}{2^N}\right)$ ($0\leq k<2^N$), the first $2^N$ Walsh functions  can be represented by columns of certain matrices of size $2^N\times 2^N$ that we refer to as \emph{Walsh--Hadamard matrices}.  Define ${\emph WH}_N$ recursively as follows. For $N=1$ let ${\emph WH}_1$ be the Haar matrix, 
$H=\frac{1}{\sqrt{2}} \left(\begin{matrix} 1& 1\\ 1& -1\end{matrix}\right)$. For $N>1$ we define ${\emph WH}_N$ as the $2^N\times 2^N$ matrix whose first $2^{N-1}$ rows form the matrix ${\emph WH}_{N-1}\otimes [1,1]/\sqrt{2}$, the Kronecker product of ${\emph WH}_{N-1}$ with $[1,1]/\sqrt{2}$, and 
whose last $2^{N-1}$ rows form the matrix ${\emph WH}_{N-1}\otimes [1,-1]/\sqrt{2}$, the Kronecker product of ${\emph WH}_{N-1}$ with $[1,-1]/\sqrt{2}$.  The matrix  ${\emph WH}_5$ is shown in Fig.~\ref{fig:dtwh_example}.
In comparison, the standard Hadamard matrices are defined as the $N$-fold Kronecker products of the Haar matrix.
Recall that $A\otimes B=\left(  \begin{matrix} a_{11} B & \cdots & a_{1n} B\\
\vdots & \ddots & \vdots \\
a_{m1} B& \cdots & a_{mn} B \end{matrix}\right)$.

The entries ${\emph WH}_N(k,n)$ of ${\emph WH}_N$ are the  normalized samples $2^{-N/2} W_n(t)$ of the Walsh functions $W_n(t)$ ($n=0,\dots, 2^N-1$) at the points $t=k/2^N$ ($k=0,\dots, 2^N-1$).  The first $2^N$ columns of ${\emph WH}_{N+1}$ are then the normalized samples of the first $2^N$ Walsh functions  at  twice the \emph{critical} rate of $2^N$ needed to distinguish the first $2^N$ Walsh functions.

Let $\mathcal{D}_N$ denote the dyadic step functions at level $N$  that are constant on the dyadic intervals  $I_{N,k}= \left[\frac{k}{2^N},\frac{k+1}{2^N}\right)$, $k=0,\dots, 2^N-1$. For $f\in\mathcal{D}_N$ one can write $f=2^{N/2}\sum_{k=0}^{2^N-1} c_k \mathbf{1}_{I_{N,k}}$. For $g\in L^2[0,1]$, the projection $P_N(g)=2^{N/2}\sum_{k=0}^{2^N-1} c_k(g) \mathbf{1}_{I_{N,k}}$ of $g$ on $\mathcal{D}_N$ is obtained by taking 
$c_k(g)=2^{N/2}\int_{I_{N,k}} g=\langle g, 2^{N/2} \mathbf{1}_{I_{N,k}}\rangle$.  
This normalization results in $c_\ell=\delta_{k,\ell}$ when $g= 2^{N/2} \mathbf{1}_{I_{N,k}}$, which has unit norm in $L^2[0,1]$. The Walsh functions $W_n(t)$, $n=0,\dots, 2^{N}-1$ form an orthonormal basis for $\mathcal{D}_N$. For  $f=2^{N/2}\sum_{k=0}^{2^N-1} c_k \mathbf{1}_{I_{N,k}}\in \mathcal{D}_N$, one can write
\[\langle f, W_n\rangle = ({\emph WH}_N^T \mathbf{c})_n, \quad n=0,\dots 2^N-1
\]
where $\mathbf{c}=[c_0,\dots, c_{2^N-1}]^T$.  The Walsh--Fourier expansion 
\[f(t)=\sum_{n=0}^{2^{N}-1} \langle f, W_n\rangle W_n(t)=\sum_{n=0}^{2^{N}-1}  (({\emph WH}_N)^T \mathbf{c})_n W_n(t)
\]
of this $f$ evaluated at $t=k/N$ yields
\[c_k =2^{-N/2}\sum_{k=0}^{2^{N}-1} \langle f, W_n\rangle W_n\left(\frac{k}{2^N}\right)
= (( {\emph WH}_N^T\mathbf{c})^T {\emph WH}_N)_k \, .
\]

\subsection{Partial sum operators and truncated WH-matrices} C. Fefferman's approach to almost everywhere convergence of Fourier series involved
proving a bound of the form $\sup_\Phi \|S_\Phi f\|_{L^1}\leq C\|f\|_{L^2}$ on $[0,1]$ where $(S_\Phi f)(x)=\sum_{n=0}^{\Phi(x)-1} \widehat{f}[n] e^{2\pi i nt}$ and $\Phi:[0,1]\to\mathbb{N}_+$ is a measurable, positive integer-valued \emph{truncation map}. The operator 
$S_\Phi$ is linear and is called a \emph{linearized partial sum operator}.   One can define analogues in the Walsh setting, namely
\begin{equation} \label{eq:sphi} (S_\Phi f)(t)=\sum_{n=0}^{\Phi(t)-1}  \langle f, W_n\rangle W_n(t)\, .
\end{equation}
We will restrict to bounded truncation maps $\Phi$. 
Since $\langle f,\, W_n\rangle=0$ when $f\in\mathcal{D}_N$ and $n>2^N$, by (\ref{eq:sphi}), when operating on $\mathcal{D}_N$
one can assume that $\Phi$ is bounded by $2^N$. However, $S_\Phi$ need not preserve $\mathcal{D}_N$ if $\Phi$ is not constant on the intervals
$I_{N,k}$.
Thus we set  $\Phi_N(k)=\min\{2^N,\langle \Phi, 2^{N/2}\mathbf{1}_{N,k}\rangle\}$ which has domain  in $\{0,\dots, 2^N-1\}$ (indices of columns of a $2^N\times 2^N$ matrix) and range in $\{1,\dots, 2^N\}$ (number of nonzero rows in a given column). 
It then suffices to evaluate $S_{\Phi_N} f$ at points $k/2^N$.
 Thinking of $k$ as indexing the rows of ${\emph WH}_N$, for $f=2^{N/2}\sum_{k=0}^{2^N-1} c_k \mathbf{1}_{I_{N,k}}$ 
one can write
\[(S_{\Phi_N} f)\left(\frac{k}{2^N}\right) =\sum_{n=0}^{\Phi_N(k)-1}2^{-N/2} W_n\left(\frac{k}{2^N}\right)  \langle f, W_n\rangle 
=(W_{N,\Phi}^T {\it WH}_N^T \mathbf{c})_k\, 
\]
where the matrix $W_{N,\Phi} $ is defined by
\[ W_{N,\Phi}(n,k)=\begin{cases} {\emph WH}_N(n,k) & n< \Phi_N(k)\\ 0 & n\geq \Phi_N(k) \end{cases}
\]
We refer to $W_{N,\Phi} $ as a \emph{truncated Walsh--Hadamard} (TWH) matrix in which, for each column $k$, the entries
are set to zero for any row index larger than $\Phi_N(k)$.  We have somewhat artificially reversed rows and columns here, in comparison with $\emph{WH}_N$,
to facilitate ensuing discussion of concepts involving height or depth of entries.
When it is clear from context that $N$ is fixed and that the truncation map $\Phi$ is in $\mathcal{D}_M$, $\Phi\leq 2^M$ with $M\leq N$ so that the local average $\Phi_N=\Phi$,  we will simply write $W_{\Phi} $. We will refer to the number $\Phi(k)$ of nonzero entries in a column $g_k$ of a ${\rm TWH}$ matrix $W_\Phi$
as the \emph{truncation length} of $g=g_k$, denoted $\ell(g)$  (we assume $\ell(g)\geq 1$). 
To distinguish rows from columns we will write $(g)_n$ for the entry in the $n$th row of column $g$ and will write
$g_k$ when referring to the index $k$ of a column.
If,  in fact,   $\Phi\in\mathcal{D}_N$ and $\|\Phi\|_\infty\leq 2^N$
then the range of $S_\Phi$ in  (\ref{eq:sphi})  is  contained in $\mathcal{D}_N$ and
 $S_\Phi$ maps $\mathcal{D}_N$ to itself as a subspace of $L^2[0,1]$.

\begin{figure}[ht]
\begin{center}
\noindent
 \includegraphics[width=12cm,height=4cm]{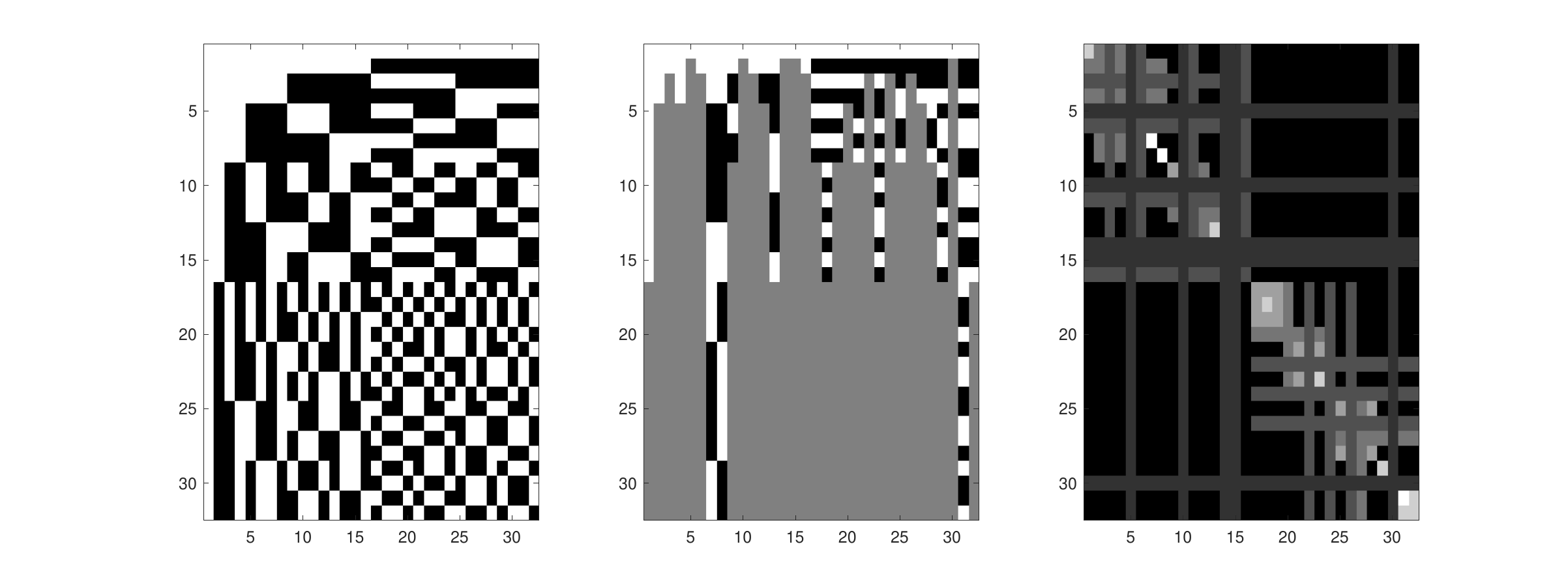}
\caption{\label{fig:dtwh_example} Left: Walsh--Hadamard matrix $\emph{WH}_N$ of size $32\times 32$ ($N=5$). Middle: DTWH matrix
of size $32\times 32$. Column truncation length $\ell=2^r$ where  $r\in \{0,\dots, 5\}$ is uniformly randomly generated.
Right: Correlation matrix of DTWH matrix in the middle}
\end{center}
\end{figure}

A consequence of the Carleson--Hunt theorem in the case of Walsh functions, expressed in C. Fefferman's terms, is that
$\sup_\Phi \|S_\Phi\|_{L^2\to L^2}$ is finite (C. Fefferman actually proved $\sup_\Phi \|S_\Phi\|_{L^2\to L^1}$ is finite). However, no existing method provides a direct bound on this quantity. Our goal in what follows is to specify such a bound over a very special subclass of functions $\Phi$ and outline a justification for the bound. Denote by $2^{\mathbb{N}}=\{1,2,4,8,\dots\}$, the nonnegative integer powers of two.

\begin{conjecture} \label{thm:uniform_bound} Let $D$ denote the collection of measurable functions $\Phi:[0,1]\to 2^{\mathbb{N}}$ such that $\Phi\in\mathcal{D}_M$ and $\Phi\leq 2^M$ for some $M\in\mathbb{N}$.  For $S_\Phi$ defined as in (\ref{eq:sphi}) one has 
$\sup_{\Phi\in D}\|S_\Phi\|_{L^2\to L^2}\leq 1+\frac{\sqrt{2}}{2}$.
\end{conjecture}

The bound of the conjecture provides the (known) almost everywhere convergence of dyadic partial sums of Walsh--Fourier series on $L^2[0,1]$ by an argument similar to the one given by C. Fefferman in \cite{fefferman_ae_1973}. The bound  $1+\frac{\sqrt{2}}{2}$ arises through specific TWH matrices that we will call \emph{standard truncations} 
of size $2^N\times 2^N$, denoted $W_N^{\rm opt}$ and described below.  The corresponding TWH matrix when $N=5$ is the left matrix in Fig.~\ref{fig:opt_biopt_N5}. The norm of the standard truncation as a function of $N$ is plotted in Fig.~\ref{fig:optnorms}. 

\begin{figure}[ht]
\begin{center}
\noindent
 \includegraphics[width=6cm,height=4cm]{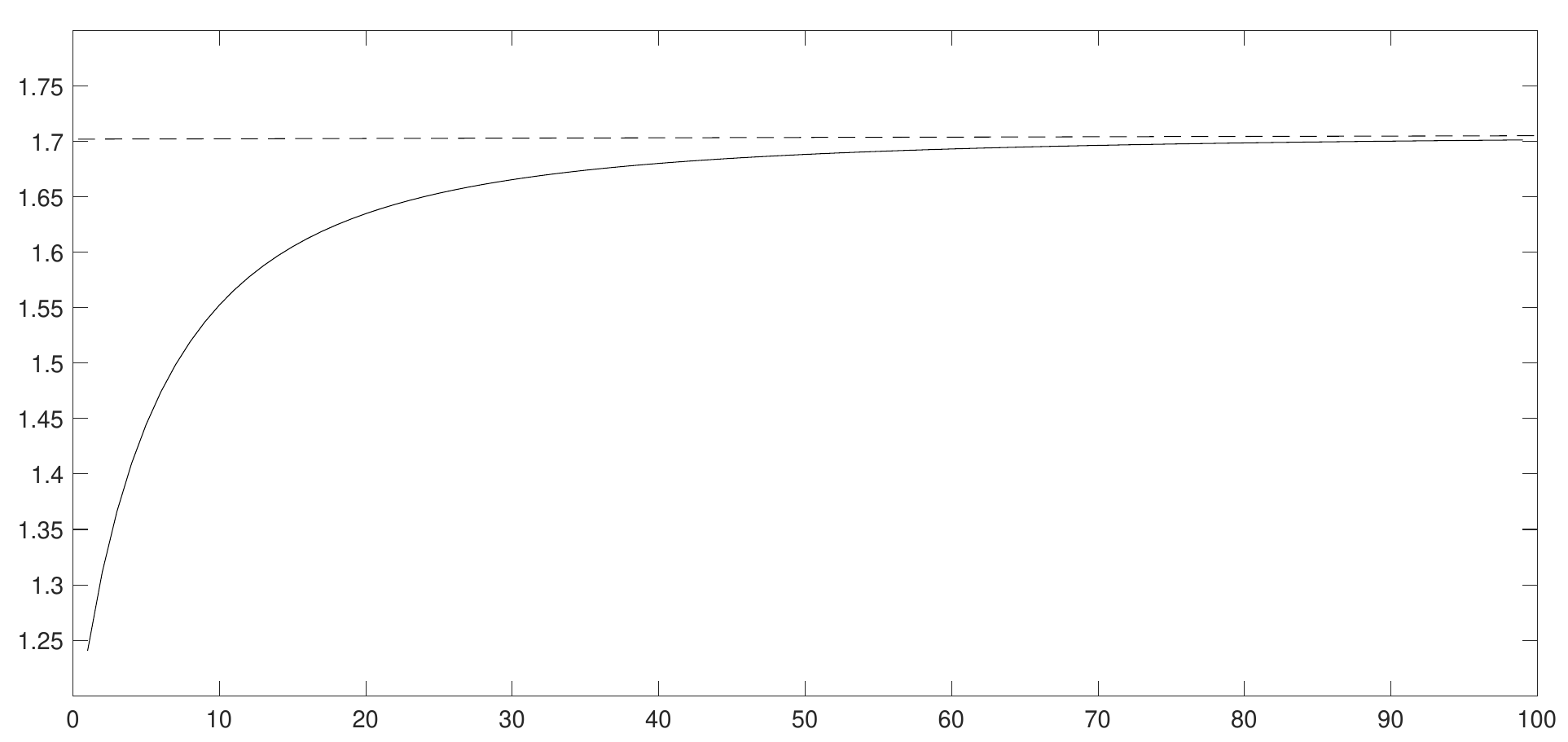}
\caption{\label{fig:optnorms} Norm of the standard truncation matrix $W_N^{\rm opt}$ as a function of $N$. The limiting value as $N\to \infty$ appears to equal
$1+\frac{\sqrt{2}}{2}=1.7071\dots$}
\end{center}
\end{figure}

Since $S_{\Phi_N}$ is represented by $W_{N,\Phi}^T{\it WH}_N^T $ and since $\|{\it WH}_N\|=1 $, one has $\|S_{\Phi_N}\|\leq \|W_{N,\Phi}\|$ where $\|\cdot\|$ is the operator norm on the appropriate $L^2$ space. Implicit in the statement of the conjecture are the facts that dyadic step functions are dense in $L^2[0,1]$, that every dyadic step function is in $\mathcal{D}_N$ for $N$ sufficiently large (depending on $f$), and that $S_\Phi f =S_{\Phi_N} f$ when $N$ is large enough and $f\in\mathcal{D}_N$.  The density of dyadic step functions and the fact that a bounded linear operator on a dense subspace extends to a bounded linear operator on the full space are the only aspects of this work that do not reduce in one way or another to finite dimensional matrix theory (and calculus).  Conjecture \ref{thm:uniform_bound} then follows from dominating, in a suitable sense, any 
$S_{\Phi_N}\sim W_{N,\Phi}^T{\it WH}_N^T $, $\Phi_N\in D$, by replacing $\Phi_N$
 by a subtruncation of the standard truncation which we denote as $\Phi_N^{\rm opt}$ defined as follows: 
$\Phi_N^{\rm opt}(0)=2^N$, $\Phi_N^{\rm opt}(1)=2^{N-1}$
and $\Phi_N^{\rm opt}(k)=2^{N-K}$ if $k\in\{2^{K-1},\dots, 2^K-1\}$.  The bound on $\|S_\Phi\|$ is stated
as a conjecture but we provide numerical evidence supported by heuristic arguments that make the conjecture highly plausible.
It should be pointed out here that coordinatewise multiplication of the columns of the Walsh--Hadamard matrix $\emph{WH}_N$ by $2^Nh$, where
$h$ is a fixed column of $\emph{WH}_N$, defines a permutation of the columns of $\emph{WH}_N$. This is a special property of Walsh functions
whose values are in $\{1,-1\}$. Likewise, multiplication of the columns of 
$W_{N,\Phi}$ by 
$2^N h$ results in a new ${\rm TWH}$ matrix $W_{N,h\Phi}$ that is \emph{equivalent} to $W_{N,\Phi}$ in the sense that 
the inner product of any pair of columns in $W_{N,h\Phi}$ is equal to that of the corresponding pair in $W_{N,\Phi}$. We will use this notion of equivalence below. 
Throughout what follows we will use  the following simple fact extensively.
\begin{lemma}\label{lem:dyadic_orthgonal} Let $\Phi\in D$. Let $h,g$ be a pair of columns of $W_\Phi$. Then either $\langle h,g\rangle=0$  or
 $(h)_n=(g)_n$ for $0\leq n<\min(\ell(h),\ell(g))$.
\end{lemma}
In other words, two columns of a dyadically truncated Walsh--Hadamard matrix are either orthogonal or are equal up to the shorter truncation length of the pair.

\begin{proof} $W_\Phi$ is a dyadic truncation of $\emph{WH}_N$ for some $N$. Because of the recursive structure of $\emph{WH}_N$, the first $2^k$ entries of a column
of $\emph{WH}_N$ is $2^{(k-N)/2}$ times a column of $\emph{WH}_k$. Since $\emph{WH}_k$ is an orthogonal matrix, for two different columns of $\emph{WH}_N$, either they are equal in the first $2^k$ rows or their truncations to the first $2^k$ rows are orthogonal to one another.
\end{proof}

 In Sect.~\ref{sect:discussion} we will briefly outline an approach to bound $\sup_\Phi \|S_\Phi\|$ in the general case of truncations that are (bounded) measurable, but not necessarily dyadic.
 Identifying a truncation of optimal norm of size $2^N\times 2^N$ for any given $N$ is intractable
but a \emph{dilation} approach that compares a truncation for a given $N$ with one for  $N+1$ that has less branching as outlined in Sect.~\ref{sect:discussion} may lead to a proof of the following.
\begin{conjecture} \label{conj:uniform_bound_general} Let $M$ denote the collection of bounded measurable functions $\Phi:[0,1]\to \mathbb{N}$. One has 
$\sup_{\Phi\in M}\|S_\Phi\|_{L^2\to L^2}\leq 2+\frac{\sqrt{2}}{2}$.
\end{conjecture}

\section{Norm bounds for dyadically truncated Walsh--Hadamard matrices\label{sect:properties}}

\subsection{Branching  in DTWH matrices}

\begin{definition} \label{def:standard_truncation} (i) The standard Walsh--Hadamard truncation of size $2^N\times 2^N$, denoted $W_N^{\rm opt}$, corresponds to the truncation that sets all entries of a column of $WH_N$ equal to zero starting from and below the first occurrence of a negative entry in that column. For the column with index $j$, if $2^k\leq j<2^{k+1}$ then the $j$th column of $W_N^{\rm opt}$ has zeros starting with the $2^{N-k-1}$-st row.

(ii) The standard two-branch Walsh--Hadamard truncations of size $2^N\times 2^N$, denoted $ B_{N-1,K}$, are 
\begin{equation}\label{eq:two_branch} B_{N-1,K}=[W_{N-1}^{\rm opt}\otimes \left(\begin{matrix}1/\sqrt{2}\\ 1/\sqrt{2}\end{matrix}\right) | \,  W_{K}^{\rm opt}\otimes\mathbf{a}(N-K) |\, T_{N-1,K}]
\end{equation}
where $\mathbf{a}(N)$ is the alternating vector $[1,-1,1,-1,\dots, 1,-1]^T/2^{N/2}\in\mathbb{R}^{2^N}$ 
and $T_{N-1,K}$ is the matrix of size $2^N\times (2^{N-1}-2^K)$ whose top entry in each row is $2^{-N/2}$ and remaining entries are zeros.
\end{definition}

The alternating vector $\mathbf{a}(N)$  is the $2^{N-1}$st column of $WH_N$.  We refer to the left half (first $2^{N-1}$ columns) of $B_{N-1,K}$ as the \emph{primary branch}
of $B_{N-1,K}$ and the next $2^K$ columns as the \emph{secondary branch} (see Defn.~\ref{def:branch}).
The matrices $W_5^{\rm opt}$ and $B_{4,2}$ are illustrated in  Fig.~\ref{fig:opt_biopt_N5}.
The ornament ``{\rm opt}'' of  $W_N^{\rm opt}$ is intended to suggest that the standard truncation has largest $\ell^2\to \ell^2$ norm among all \emph{dyadically} truncated Walsh--Hadamard matrices of size $2^N\times 2^N$ (up to equivalence).  The rest of the presentation outlines an approach to proving this claim, though an actual proof is not provided here.

The concept of branching here plays an important role in explaining why Conj.~\ref{thm:uniform_bound} should
hold. We formalize it for dyadic truncations as follows.

\begin{definition} \label{def:branch} Let $\Phi: \{0,1,\dots, 2^N-1\}\to 2^{\{0,\dots, N\}}$. 
A subset of mutually nonorthogonal columns of $W_\Phi$ is called a branch. 
A branch is said to be complete if $\Phi$ is maximal, that is, for each index $j$ of a column $g_j$ in the branch, its length $\Phi(j)$ is maximal with respect to the property of $g_j$ being nonorthogonal to other columns in the branch.
 A set of columns that does not form a branch is said to be bifurcated. $\Phi$ has a node at level $L$ if there are mutually orthogonal columns  $g_j,g_k$ with $\min\{\Phi(j), \Phi(k)\}>2^L$ that are equal in the first $2^L$ entries but $(g_j)_{2^L}\neq (g_k)_{2^L}$.
\end{definition}

The matrix $W_5^{\rm opt}$ on the left in Fig.~\ref{fig:opt_biopt_N5} has only a single branch whereas $B_{4,2}$ on  the right has two branches.
A column can lie on more than one branch as is the case for the minimal length columns of the matrix $B_{4,2}$. If a column $g$
lies on two branches then there must be a node at some level $L$ and $\ell(g)\leq 2^L$. In this case we say that $g$ is \emph{at or below the node} whereas columns that lie on only one of the branches are \emph{above the node}.
A column can only lie on one complete branch.  We say that a sub-branch (subset of columns of a branch) is complete if
each column has maximal length relative to the other columns in the sub-branch.
The standard truncation $W_N^{\rm opt}$ is characterized, up to equivalence, 
as a complete, single-branch truncation.

\begin{figure}[ht]
\begin{center}
\noindent
 \includegraphics[width=12cm,height=4cm]{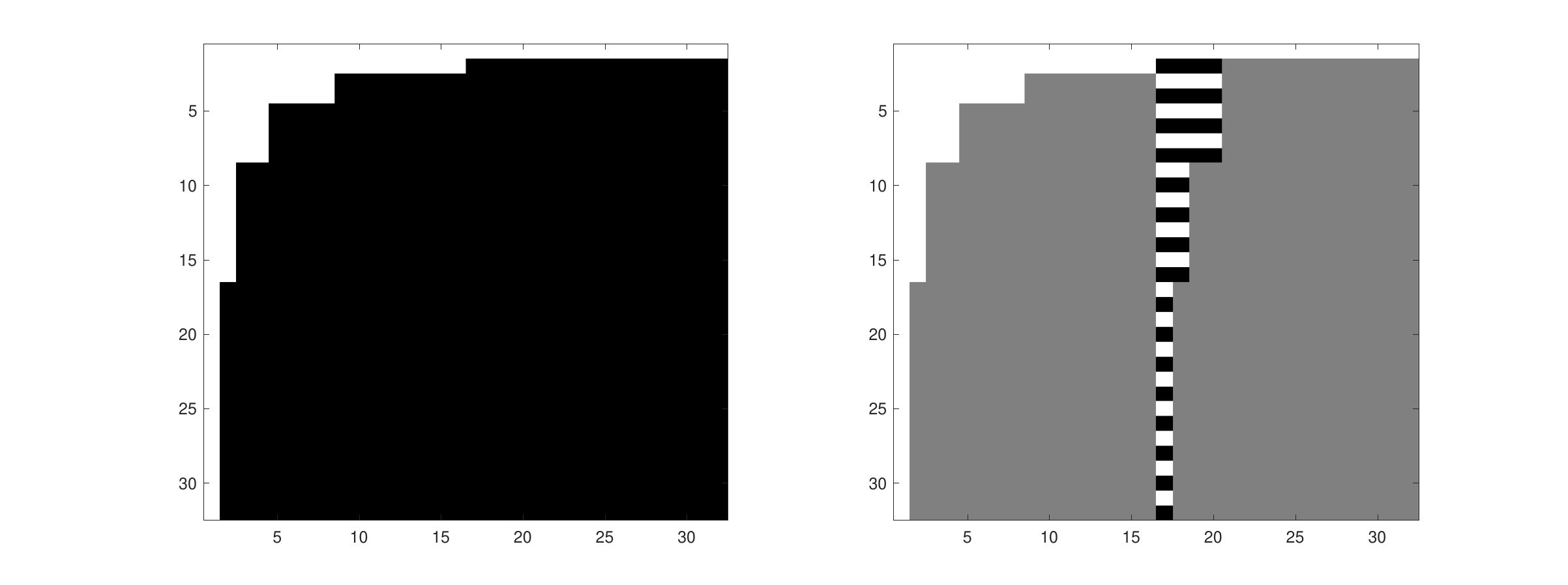}
\caption{\label{fig:opt_biopt_N5} Left: The $32\times 32$ matrix $W_N^{\rm opt}$ ($N=5$). Right: $32\times 32$ two-branch matrix
$B_{N-1,K}$ with $N=5$ and $K=2$}
\end{center}
\end{figure}

\subsection{Branching reduces norm in DTWH matrices}

\begin{conjecture}
\label{claim:one_node} Suppose that $\mathcal{C}$ is a set of columns of a DTWH matrix $W_\Phi$ that has exactly one node in which each branch is complete above the node. Denote by $W_{\Phi,\mathcal{C}}$ the submatrix of $W_\Phi$ whose columns are those in $\mathcal{C}$. Let $\Phi'$ be the truncation that assigns to each column at or above the node its length with respect to a selected primary branch, and $\Phi'=\Phi$
for any column whose length is smaller than the nodal length.  Then $\|W_{\Phi',\mathcal{C}}\|\geq \|W_{\Phi,\mathcal{C}}\|$.
\end{conjecture}

\begin{proposition}\label{prop:strong_dyad} If Conjecture~\ref{claim:one_node} holds then $W_{N}^{\rm opt}$ has largest norm
among all 
DTWH matrices of size $2^N\times 2^N$.
\end{proposition}

\begin{proof}
Assume that Conj.~\ref{claim:one_node} holds in the $2^N\times 2^N$ case. Suppose that $W_\Phi$ is a bifurcated DTWH matrix of size $2^N\times 2^N$.  Then there is a deepest nodal level $L$, that is, such that $W_\Phi$ does not have any nodal levels larger than $L\leq N$.  Fix such $L$. The columns of $W_\Phi$ can be segmented into $2^{L}$ blocks of $2^{N-L}$ consecutive columns such that any pair of columns $g_1,g_2$ in different blocks $\mathcal{B}_1,\mathcal{B}_2$ with $\ell(g_i)>2^L$ are orthogonal to one another.  If $\mathcal{B}$ is such a block that has a node at level $L$ then we can apply Conj.~\ref{claim:one_node} to the submatrix whose columns $\mathcal{C}$  are those of $\mathcal{B}$ supplemented by any columns below the node that are branch-aligned with $\mathcal{B}$ (that is, any column $\tilde{g}$ such that $\ell(\tilde{g})\leq 2^L$ and $\langle \tilde{g},g\rangle= \ell(\tilde{g})/2^N$ for any $g$ in $\mathcal{B}$).  
Define a \emph{node-reduced}  truncation $\Phi'$ as follows: (i) replace the restriction of $\Phi$ to $\mathcal{B}$ by $\Phi'$ whose restriction to $\mathcal{B}$ is equivalent to the restriction of the standard truncation and (ii) for columns $g$ outside $\mathcal{B}$, set $\Phi(g)=\Phi(g)$.
By Conj.~\ref{claim:one_node}, 
$\|W_{\Phi',\mathcal{C}}\|\geq \|W_{\Phi,\mathcal{C}}\|$.
 We claim that, in fact, $\|W_{\Phi'}\|\geq \|W_{\Phi}\|$.
  Any length modification $\Phi'$ such that $\Phi'(g)\geq 2^L$   whenever $\Phi(g)\geq 2^L$ 
  (and $\Phi'(g)=\Phi(g)$ if $\Phi(g)<2^L$)
  does not change inner products of pairs in which one column was above the node and the other below. As described below, the sum of coefficients $c_g$ of vectors $g$ above the node for an input vector $\mathbf{c}$ that maximizes $\|W_{\Phi,\mathcal{B}}\mathbf{c}_{\mathcal{B}}\|$ (with $\mathbf{c}_{\mathcal{B}}$ the restriction of $\mathbf{c}$ to coordinates in ${\mathcal{B}}$)  is optimized when the columns above the node form a single branch. Thus, the net contribution to $\|W_{\Phi} \mathbf{c}\|^2$ coming from terms of the form $\langle c_{g_1} g_1,c_{g_2} g_2\rangle$ where $\Phi(g_1)<2^L$ and $\Phi(g_2)\geq 2^L$  is not decreased when the bifurcated columns above the node are replaced by a single branch within  a fixed block $\mathcal{B}$ and the coordinates of $\mathbf{c}$ outside $\mathcal{B}$ are fixed, even if there are other bifurcations at the same level. This explains why $\|W_{\Phi'}\|\geq \|W_{\Phi}\|$ for the node-reduced truncation $\Phi'$.
  Repeating this for any other node at deepest level $L$  results in a new DTWH matrix whose norm is at least that of $W_\Phi$ and  whose deepest node  level is less than $L$. Since there are finitely many levels, the proposition  follows.
\end{proof}

\bigskip

\section{Evidence for optimality of the standard truncation\label{sect:evidence}}

In this section we consider the following special case of Conj.~\ref{claim:one_node}. 

\begin{conjecture}
\label{claim:bnk_opt} For fixed $N$, for each $K=0,\dots, N$, the norm of the matrix $B_{N-1,K}$ is less than that of 
$W_{N}^{\rm opt}$.
\end{conjecture}

Conjectures~\ref{claim:one_node} and  \ref{claim:bnk_opt} are listed as such because we do not provide actual proofs here, but we will provide adequate evidence to support the latter.
Conjecture \ref{claim:bnk_opt} is the special case in which the node is at level zero so the columns truncated to length one are nodal columns, the primary branch is the set of all columns of index smaller than $2^{N-1}$ and the secondary branch consists of the columns with indices $2^{N-1},\dots, 2^{N-1}+2^K-1$.   
It asserts, in this specific case,  
that the norm of a truncation with one node is dominated by replacing the truncation lengths of the columns above the node by ones in which the same columns now lie on a single branch. 
For other one-node  DTWH matrices,  for the part $\mathcal{C}'$ of $\mathcal{C}$ above the node, we can assume that $\mathcal{C}'$ forms a block of $2^{N-L}$ 
consecutive columns of $W_\Phi$ of the form $W_{N-L,\Phi'}\otimes g$ 
for a column $g$ of the $2^L\times 2^L$ Walsh--Hadamard matrix $\emph{WH}_{L}$.  Columns on different branches that are above the node  are orthogonal to one another by Lem.~\ref{lem:dyadic_orthgonal}. The norms attached to each of these branches are then optimized when the parts of the branches above the node are complete.
If the part of a branch above the node has $2^{N-\tilde{L}}$ columns ($L\leq \tilde{L}$) and is complete, then that part is equivalent to 
$W_{N-\tilde{L}}^{\rm opt}\otimes \tilde{g}$ where $\tilde{g}$ is a column of $\emph{WH}_{\tilde{L}}$.
The case of $B_{N-1,K}$ thus  typifies the situation of the part of $\mathcal{C}$ 
\emph{above the node} (that is the columns in $\mathcal{C}$ such that $\ell(g)\geq 2^L$) in  Conj.~\ref{claim:one_node}.

We  have not explained why we consider only the case in which the number of columns in a branch is $2^K$ for some $K$
and in which the number of columns in the primary branch is $2^{N-1}$. 
 In light of arguments outlined below, cases in which the number of columns $s$ in the secondary  branch is not a power of two can be viewed as intermediate between $B_{N-1,K-1}$ and $B_{N-1,K}$ when $2^{K-1}<s<2^K$, and satisfy a corresponding norm inequality.  We also claim that if $B_{P,K}$ is a two-branch truncation with nodal level $L=0$  whose primary branch has width $2^P$ where $K\leq P<N-1$ then 
 $\|B_{P,K}\|\leq \|B_{N-1,K}\|$.
 With these observations, establishing norm bounds for $B_{N-1,K}$ can be viewed as the essence of establishing 
 Conj.~\ref{claim:one_node}.

\smallskip
In support of  Conj.~\ref{claim:bnk_opt} we argue, specifically,  that the operator norm of $B_{N-1,K}$ decreases with $K$ to conclude that  $\|B_{N-1,K}\|\leq \|W_N^{\rm opt}\|$.  The approach is to express the value of $F$ in
(\ref{eq:falphabeta}) at a critical point (as a function of $(\alpha,\beta)\in (0,1)^2$)  in terms of an expression that can be argued to decrease with $K$ for fixed $N$ (\ref{eq:Fcritical}), followed by arguments explaining this decrease.

Observe that a unit
vector $\mathbf{u}\in\mathbb{R}^{2^N}$ that optimizes $\|B_{N-1,K} (\mathbf{u})\|$ will have the form 
$\mathbf{u}=\alpha\mathbf{x}+\beta\mathbf{y}+\gamma\mathbf{z}$ where the unit vector $\mathbf{x}\in\mathbb{R}^{2^{N}}_+$ (the ``$+$'' signifies that all coordinates are nonnegative) is supported in
the first $2^{N-1}$ coordinates, the unit vector $\mathbf{y}\in\mathbb{R}^{2^N}_+$ is supported in the next $2^K$ coordinates, and $\mathbf{z}$ is supported in the last $2^{N-1}-2^K$ coordinates
and has constant entries $(2^{N-1}-2^{K})^{-1/2}$ in each of those coordinates. Since $\|\mathbf{u}\|=1$,  $\alpha^2+\beta^2+\gamma^2=1$. Then
\begin{eqnarray}\label{eq:bnk_norm} 
\|B_{N-1,K} (\mathbf{u})\|^2
&=&\alpha^2 \|W_{N-1}^{\rm opt}\mathbf{x}\|^2+\beta^2 \|W_{N,K}^{\rm opt} \mathbf{y}\|^2 +\gamma^2 \|T_{N-1,K}\mathbf{z}\|^2 \notag \\
&+&2\gamma \langle T_{N-1,K}\mathbf{z},\alpha W_{N-1}^{\rm opt}\mathbf{x}+\beta W_{N,K}^{\rm opt}\mathbf{y}\rangle\, .
\end{eqnarray}
Here, $W_{N,K}^{\rm opt}=W_K^{\rm opt}\otimes\mathbf{a}(N-K)$ and $T_{N-1,K}$ are as defined in  (\ref{eq:two_branch}). Because of the structure of $T_{N-1,K}$ and $\mathbf{z}$,   
\[\langle T\mathbf{z},\alpha W_{N-1}^{\rm opt}\mathbf{x}
+\beta W_{N,K}^{\rm opt}\mathbf{y}\rangle =2^{-N} (2^{N-1}-2^{K})^{1/2}(\alpha \|\mathbf{x}\|_1+\beta\|\mathbf{y}\|_1)\]
where $\|\mathbf{x}\|_1$ is the sum of the coordinates of $\mathbf{x}\in\mathbb{R}^{2^N}_+$.
We may rewrite (\ref{eq:bnk_norm}) as 
\begin{eqnarray}
\label{eq:bnk_norm2}
\|B_{N-1,K} (\mathbf{u})\|^2
&=&\alpha^2 \|W_{N-1}^{\rm opt}\mathbf{x}\|^2+\beta^2 \|W_{N,K}^{\rm opt} \mathbf{y}\|^2 +\gamma^2 2^{-N}(2^{N-1}-2^K)\notag \\
&+&2^{1-N} (2^{N-1}-2^{K})^{1/2}\gamma (\alpha\|\mathbf{x}\|_1+\beta \|\mathbf{y}\|_1) \, .
\end{eqnarray}
Considered as a function with inputs $\mathbf{x},\mathbf{y}$, $\alpha,\beta$ (and $\gamma^2=1-\alpha^2-\beta^2$), abbreviating 
$x=2^{-N/2}\|\mathbf{x}\|_1$ and 
$y=2^{-N/2}\|\mathbf{y}\|_1$, for $N$ and $K$ fixed, $\|B_{N-1,K} (\alpha\mathbf{x}+\beta\mathbf{y}+\gamma\mathbf{z})\|^2$ can be expressed in the more generic form
\begin{equation}\label{eq:falphabeta} F(\alpha,\beta,x,y)=\alpha^2 A^2+\beta^2 B^2+\gamma^2 C^2 +2\gamma C(\alpha x+\beta y)
\end{equation}
with the specific associations made in Tab.~\ref{tab:f-params}.

\begin{table}[tbhp]
{\footnotesize
\caption{Parameters of  $F$ in (\ref{eq:falphabeta}) when $F=\|B_{N-1,K} (\alpha\mathbf{x}+\beta\mathbf{y}+\gamma\mathbf{z})\|^2$}
\label{tab:f-params}
\small
\begin{center}
\begin{tabular}{|c|c|c|c|c|c|}
\hline 
parameter of $F$  & $A^2$  & $B^2 $ & $C^2 $& $x$ & $y  $ \\ \hline
term of $\|B_{N-1,K}(\cdots)\|^2$ & $\|W_{N-1}^{\rm opt} \mathbf{x}\|^2$ &$\|W_{K}^{\rm opt} \mathbf{y}\|^2$ & $2^{-1}-2^{K-N}$ &
$\frac{\|\mathbf{x}\|_1}{\sqrt{2^N}}$ &$\frac{\|\mathbf{y}\|_1}{\sqrt{2^N}}$ \\ \hline
\end{tabular}
\end{center}
}
\end{table}%

Momentarily we will consider behavior of $F$ near a critical point. First, before analyzing $F$ all at once, we review separately optimization of the ``$\ell^2$'' and ``$\ell^1$'' parts of $F$ in order to gain some perspective on relative behavior of the parameters of $F$.

\paragraph{Analysis of $F=F_2+F_1$}

One can express $F$ in  (\ref{eq:bnk_norm}) and  (\ref{eq:falphabeta}) as $F=F_2+F_1$ where
\begin{equation}\label{eq:F_2} F_2(\alpha,\beta;\mathbf{x},\mathbf{y})=\alpha^2\|W_{N-1}^{\rm opt}\mathbf{x}\|^2+\beta^2\|W_{K}^{\rm opt}\mathbf{y}\|^2+
\gamma^2 \left(\frac{1}{2}-2^{K-N}\right)
\end{equation}
and
\begin{equation}\label{eq:F_1} F_1(\alpha,\beta;\mathbf{x},\mathbf{y})= 2^{-N/2}\left(\frac{1}{2}-2^{K-N}\right)^{1/2}\gamma(\alpha \|\mathbf{x}\|_1+\beta\|\mathbf{y}\|_1)\, .
\end{equation}
Although $F_1$ and $F_2$ are separately maximized at different $(\alpha,\beta)$-parameter values, understanding their separate optima can aid in understanding how optimizing parameters for $B_{N-1,K}$ change with $K$.
For fixed $\mathbf{x},\mathbf{y}$, $F_2$ considered \emph{by itself} is maximized when the largest term of the three terms in (\ref{eq:F_2}) is maximized. Since $\|W_N^{\rm opt}\|$ is increasing in $N$, when $\mathbf{x}$ is close to a norm eigenvector of $W_N^{\rm opt}$, this happens when $\alpha=1$ (and $\beta=\gamma=0$).
The parameter $\beta$ is coupled with $\gamma$ in $F_1$ and only indirectly coupled with $\alpha$ through $\alpha^2+\beta^2+\gamma^2=1$ so one might expect $F_1$ to give an accurate estimate of $\beta$ in $F$, at least when $K$ is small.  Again with $\mathbf{x},\mathbf{y}$ fixed, setting $r=\|\mathbf{y}\|_1/\|\mathbf{x}\|_1$ and assuming $r<1$, using the 
binomial approximation $\sqrt{1+x}\approx 1+\frac{x}{2}$  with $x=4/(1/r-r)^2$ for small $r$,
$F_1$ is maximized (with respect to $\alpha,\beta$) approximately when  
\begin{equation}\label{eq:f1params} \beta
\approx  {\alpha}\frac{r}{(1-r^2)};\quad \gamma\approx \alpha \frac{1}{\sqrt{1-r^2}}\, .
\end{equation}
The condition $\alpha^2+\beta^2 +\gamma^2=1$ ultimately 
yields, at a critical point of $F_1$,
\begin{equation*}F_1(\alpha,\beta;\mathbf{x},\mathbf{y})= 2^{-N/2}\left(\frac{1}{2}-2^{K-N}\right)^{1/2}\gamma(\alpha+\beta r) \|\mathbf{x}\|_1 
\approx 2^{-N/2}\left(\frac{1}{2}-2^{K-N}\right)^{1/2}\frac{(1-r^2)^{1/2}}{1+(1-r^2)^2} \|\mathbf{x}\|_1  \, 
\end{equation*}
each of whose factors is decreasing with $K$, as explained below.
Generally $|F_1|<1$ whereas $F_2\approx (1+\sqrt{2}/2)^2$ so $F_2$ will dominate $F_1$ when optimizing $F=F_1+F_2$.
At a critical point of $F$ the growth of $F_2$ in $\alpha$ will balance decay of $F_1$ when $\alpha$ approaches one while decay of
$F_2$ in $\beta$ will balance growth of $F_1$ in $\beta$ when $\beta $ reaches a critical threshold that grows with $K$. For small $K$
when $\alpha$ remains close to one, according to  (\ref{eq:f1params}) $\beta$ should grow slightly faster than $r$ which is proportional to $2^{K/2}$,
more specifically, $\beta\sim 2^{(K-N)/2}$ for $K$ small.

The analysis of $F_1$ by itself  provides correct order of magnitude of $\beta$ with regard to $F=F_1+F_2$  but not of $\gamma$.
For small values of $K$, as indicated in Fig.~\ref{fig:alphabetaxyz},  the decrease in $\gamma$ when $K-1$ is replaced by $K$ will be proportional to  the change in the number of columns of $B_{N-1,K}$ truncated to length one, that is, to $2^{K-1}$.   This is also explained further below.

\begin{figure}[htp]
\begin{center}
\noindent
 \includegraphics[width=12cm,height=4cm]{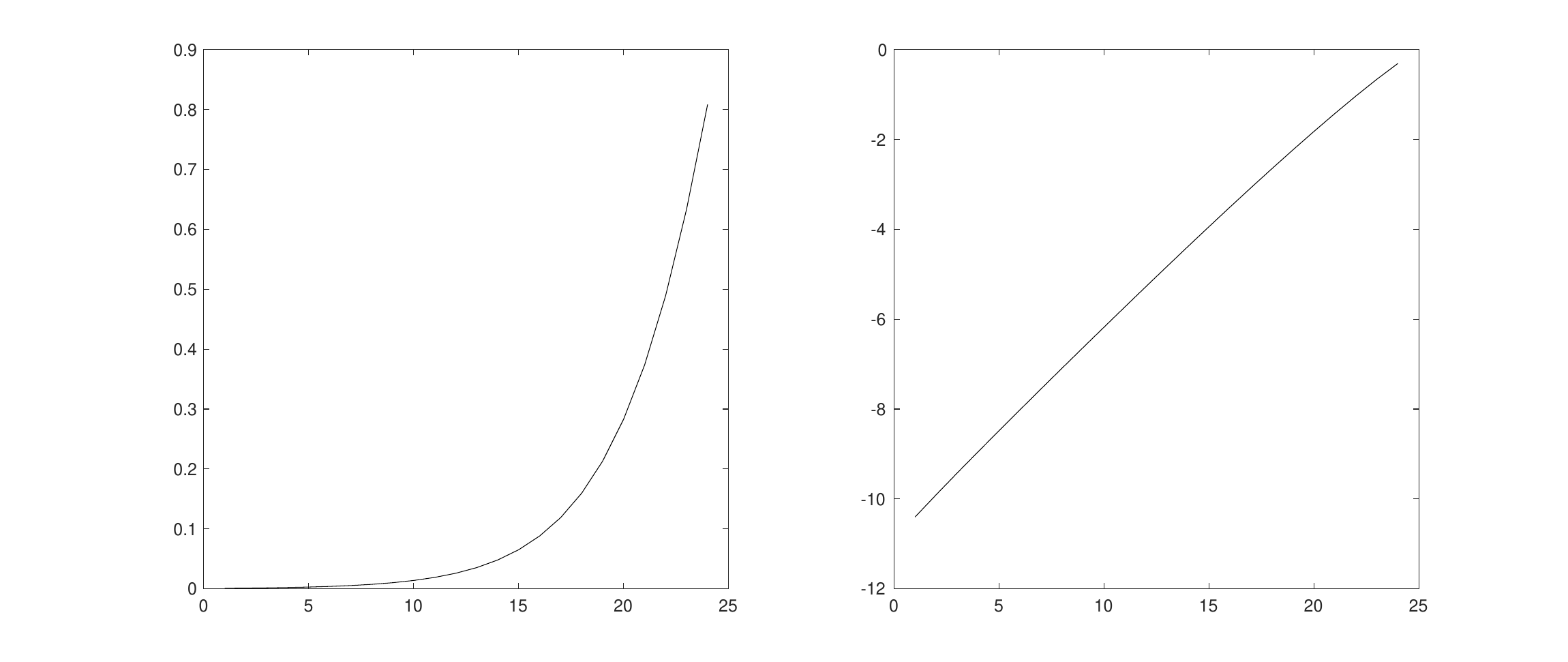}
\caption{\label{fig:secondlone_over_primelone_N25} Left: Plot of $\|\mathbf{y}\|_1/\|\mathbf{x}\|_1$ 
where $\mathbf{x}$ and $\mathbf{y}$ optimize $\|B_{N-1,K}(\mathbf{u})\|$ as in (\ref{eq:bnk_norm})
for $K=1$ to $K=23$ and $N=25$. Right: $\log_2(\|\mathbf{y}\|_1/\|\mathbf{x}\|_1)$ for same range. The ratio is approximately proportional to $2^{K/2}$ for small $K$}
\end{center}
\end{figure}

\begin{proposition} \label{prop:Fcritical} 
Let $F$ be defined as in (\ref{eq:falphabeta}) where $A,B,C$ and $x,y$ are fixed positive numbers, and for any $(\alpha,\beta)\in (0,1)^2$, $\gamma$ is the positive solution of  $\gamma^2=1-\alpha^2-\beta^2$.
 Then at a critical point $(\alpha,\beta)$ of $F$ in (\ref{eq:falphabeta}) as a function of $(\alpha,\beta)$,
  that is $\frac{\partial F}{\partial\alpha}= 0 =\frac{\partial F}{\partial\beta}$, one has
 \begin{equation}\label{eq:Fcritical}F(\alpha,\beta)=A^2+Cx\frac{\gamma}{\alpha}=B^2+Cy\frac{\gamma}{\beta}\, .
 \end{equation}
 \end{proposition}
 A proof of Prop.~\ref{prop:Fcritical}  can be found in Appendix \ref{appendix:proof_Fcritical}.

To prove Conj.~\ref{claim:bnk_opt}  it would be sufficient to show that as $K$ increases, if the inputs $\mathbf{x},\mathbf{y}$  and $\alpha,\beta$ correspond to optimizers of $B_{N-1,K}$ then  the corresponding expression $A^2+Cx\frac{\gamma}{\alpha}$ decreases, where
$A,C,x,\gamma,\alpha$ are as in  Tab.~\ref{tab:f-params}.  We start with general  arguments for plausibility of this conjecture. We will denote by $A(N-1,K)$, $C(N-1,K)$ and $x(N-1,K)$ values of $A,C,x$ in (\ref{eq:falphabeta}) corresponding to Tab.~\ref{tab:f-params}
when $N,K$ in (\ref{eq:bnk_norm}) are specified.  

\paragraph{Subclaim: $A^2(N-1,K)$ increases \emph{moderately} as $K$ ranges from null to  $K=N-1$.}
In the \emph{null} case in which there is no secondary component, the matrix
 corresponding to $B_{N-1,K}$ is $W_N^{\rm opt}$. In that case, up to normalization, $\mathbf{x}$ forms the first $2^{N-1}$
 entries of the norm eigenvector of $W_N^{\rm opt}$. This  $\mathbf{x}$ is not a norm eigenvector of $W_{N-1}^{\rm opt}$ so the value of $A^2(N-1, {\rm null})$  as defined
 in Tab.~\ref{tab:f-params} is smaller than $\|W_{N-1}^{\rm opt}\|$.
When  $K=N-1$, the left side of  $B_{N-1,N-1}$ is $W_{N-1}^{\rm opt}\otimes [1,1]^T/\sqrt{2}$ and the right side is $W_{N-1}^{\rm opt}\otimes [1,-1]^T/\sqrt{2}$. Thus the two sides are orthogonal to one another and each has norm equal to that of $\|W_{N-1}^{\rm opt}\|$. We can take  $\mathbf{x}$ to be the norm eigenvector of $W_{N-1}^{\rm opt}$ when $K=N-1$. Thus, $A^2(N-1, {\rm null})< A^2(N-1, N-1)$. 
We do not provide here an analytical proof that $A^2$ is increasing in $K$ for intermediate values. However, as suggested by Fig.~\ref{fig:nprime25_vecs},
for intermediate $K$ the $\mathbf{x}$-component of an optimizer of $\|B_{N-1,K}(\mathbf{u})\|$ can be regarded as a perturbation that gets closer to the norm eigenvector of $W_{N-1}^{\rm opt}$ as $K$ increases to $N-1$.  This plausibility argument is supported by numerical evidence in  Fig.~\ref{fig:Fsumterms} which plots $A^2(N,K)$ vs. $K$ for $N=24$. Further quantification of $A^2(N,K)$ is described below.

\paragraph{Subclaim: $x(N-1,K)=2^{-N/2}\|\mathbf{x}(N-1,K)\|_1$ decreases in $K$.}  For now we refer to Fig.~\ref{fig:nprime25_vecs} in which the ``level vector'' plots are relatively \emph{flatter}, and thus have larger $\ell^1$-norm, for smaller $K$.  Below we will quantify $x(N,0)$ in terms of the eigen-decomposition of $W_N^{\rm opt}$.  

\paragraph{Subclaim: $\gamma(N-1,K)/\alpha(N-1,K)$ is decreasing in $K$.} Log plots of $\alpha(N,K)$ and $\gamma(N,K)$ are provided in Fig.~\ref{fig:alphabetaxyz}. Further support of the subclaim is provided below.

We add to these claims the observation that 
\[C^2(N-1,K)-C^2(N-1,K+1)=2^{-1}-2^{K-N}-(2^{-1}-2^{K+1-N})=2^{K-N}\]
 and $C^2(N-1,N-1)=0$. Together these claims show that 
$A^2$ increases moderately with $K$ while $Cx\gamma/\alpha$ decreases with $K$ to zero when $K=N-1$.  To conclude from these claims that the $K={\rm null}$ value of $F$ in (\ref{eq:Fcritical}) is larger than the  $K=N-1$ value just requires that the increase in $A^2$ is smaller than the $K={\rm null}$ value of $Cx\gamma/\alpha$.  To show that the quantity is monotonically decreasing requires some refinement of approaches outlined below in  Sect.~\ref{sect:rigorous}.

We close this section with a comment on the more general case of a set of columns $\mathcal{C}$ such that $\mathcal{C}$ has a single node at level $L$ in the first $2^{N-L}$
columns. We can assume that the columns at or above the node then form a matrix of the form $B_{N-L-1,K}\otimes \mathbf{e}_L$
where $2^{N/2}\mathbf{e}_L$ has ones in the first $2^L$ rows and zeros below.  Columns in  $\mathcal{C}$ below the node then are equal to  $\mathbf{e}_J$ for some $J=0,\dots, L-1$.
 If $N_J$ is the number of  columns in $\mathcal{C} $ with truncation length $2^J$ then  $N_J\leq 2^{N-J-1}$. 
The expression corresponding to (\ref{eq:bnk_norm}) in this setting replaces  $\gamma T_{N-1,K}\mathbf{z}$  by a term of the form
$T_L\mathbf{z}$ where $T_L \mathbf{z}=\sum_{J=0}^L \gamma_J \sqrt{N_J} \mathbf{e}_J$. This resulting expression
is more complicated than  (\ref{eq:bnk_norm2}) due to  terms with different truncation lengths. 
However, an analogue of Prop.~\ref{prop:Fcritical} can be established and in principal parallel methods to those just outlined can then be used
to establish  Conj.~\ref{claim:one_node}.

 \paragraph{Eigenvectors of $W_N^{\rm opt}$}
 The $\ell^1$-norms of optimizing input vectors of (\ref{eq:bnk_norm}) appear in the optimal values of (\ref{eq:bnk_norm2})  and its general form (\ref{eq:falphabeta}).
The inputs $\mathbf{x}$ and $\mathbf{y}$ are approximate eigenvectors of $W_{N-1}^{\rm opt}$ and $W_K^{\rm opt}$ respectively.
 The structure of $W_N^{\rm opt}$
allows for explicit calculation of $\ell^1$ norms of its eigenvectors in terms of the  eigenvalue and endpoints of the corresponding eigenvectors.
 The matrix  $W_N^{\rm opt}$ is symmetric. It eigenvectors have constant entries on indices $k$ such that $2^\ell\leq k<2^{\ell+1}$.  These \emph{level entries}
 determine eigenvectors of the matrix 
 \begin{eqnarray}\label{eq:level_matrix} 
 M_N&=&2^{-N/2} D_N^{1/2} C_N D_N^{1/2}\quad ((N+1)\times (N+1))\\
D_N&=&{\rm diag}(1,1,2,4,\dots, 2^{N-1}),\quad
C_N(i,j)=\begin{cases} 1&0\leq j \leq N-i \\ 0 & {\rm else}
\end{cases}
\end{eqnarray} 
 We refer to an eigenvector  $\mathbf{c}=[c_0,\dots, c_N]^T$  of $M_N$ as a \emph{level (eigen)-vector} of $W_N^{\rm opt}$. If   $\mathbf{c}$ has $M_N$-eigenvalue $\lambda$  then the entries of $\mathbf{c}$ satisfy
 \begin{equation}\label{eq:m-eigenvector}
 \lambda c_k=2^{q(k)} \left(c_0+\sum_{j=1}^{N-k} 2^{(j-1)/2} c_j\right);\quad  \begin{cases} q(k)= -N/2 & k=0\\
q(k)=(k-1-N)/2 & k=1,\dots, N\end{cases}
 \end{equation}

 We will assume in what follows that $\lambda=\lambda(N)$ refers to the norm-eigenvalue $\lambda(N)=\|W_N^{\rm opt} \|$. 
  The eigenvector $\mathbf{x}$ of $W_N^{\rm opt}$ corresponding to the \emph{level eigenvector} $\mathbf{c}$ of $M_N$ is 
 \[\mathbf{x}=[c_0,c_1,2^{-1/2} c_2, 2^{-1/2} c_2,\dots, \underbrace{2^{(1-N)/2}c_N,\dots, 2^{(1-N)/2}c_N}_{2^{N-1}}]^T\, .\]
 From (\ref{eq:m-eigenvector}) one then has
 \begin{equation}\label{eq:level_correspondence} \|\mathbf{x}\|_1=c_0+\sum_{j=1}^N 2^{(j-1)/2} c_j =2^{N/2} \lambda c_0\, .
 \end{equation}
 Similarly, if the vector $\mathbf{y}$ in (\ref{eq:bnk_norm}) corresponds to an eigenvector of $W_K^{\rm opt}$
 then its corresponding embedding in $\mathbb{R}^{2^N}$ satisfies $\|\mathbf{y}\|_1=c_0(K)+\sum_{j=1}^N 2^{(j-1)/2} c_j(K) =2^{K/2} \lambda(K) c_0(K)$.
 This suggests that, at least for small $K$, when the vectors $\mathbf{x}$ and $\mathbf{y}$ approximately optimize $\|W_{N-1}^{\rm opt} \mathbf{x}\|$ and 
 $\|W_K^{\rm opt }\mathbf{y}\|$ respectively,  the value $r=\|\mathbf{y}\|_1/\mathbf{x}\|_1$ in (\ref{eq:f1params}) should satisfy $r\approx 2^{(K-(N-1))/2} c_0(K)\lambda(K)/(c_0(N-1)\lambda(N-1))$. We argue (see Fig.~\ref{fig:optwalsh25_normvec_and_czero}) that $c_0(N)\lesssim 1/N^p$ for some $p\in (1,2)$. The values of $\lambda(N)$ correspond to Fig.~\ref{fig:optnorms}. That $\beta(N,K)\sim 2^{(K-N)/2}$ in  (\ref{eq:bnk_norm}) as predicted by this analysis  is confirmed numerically in Fig.~\ref{fig:alphabetaxyz}.

\begin{figure}[ht]
\centering
\noindent
 \includegraphics[width=12cm,height=4cm]{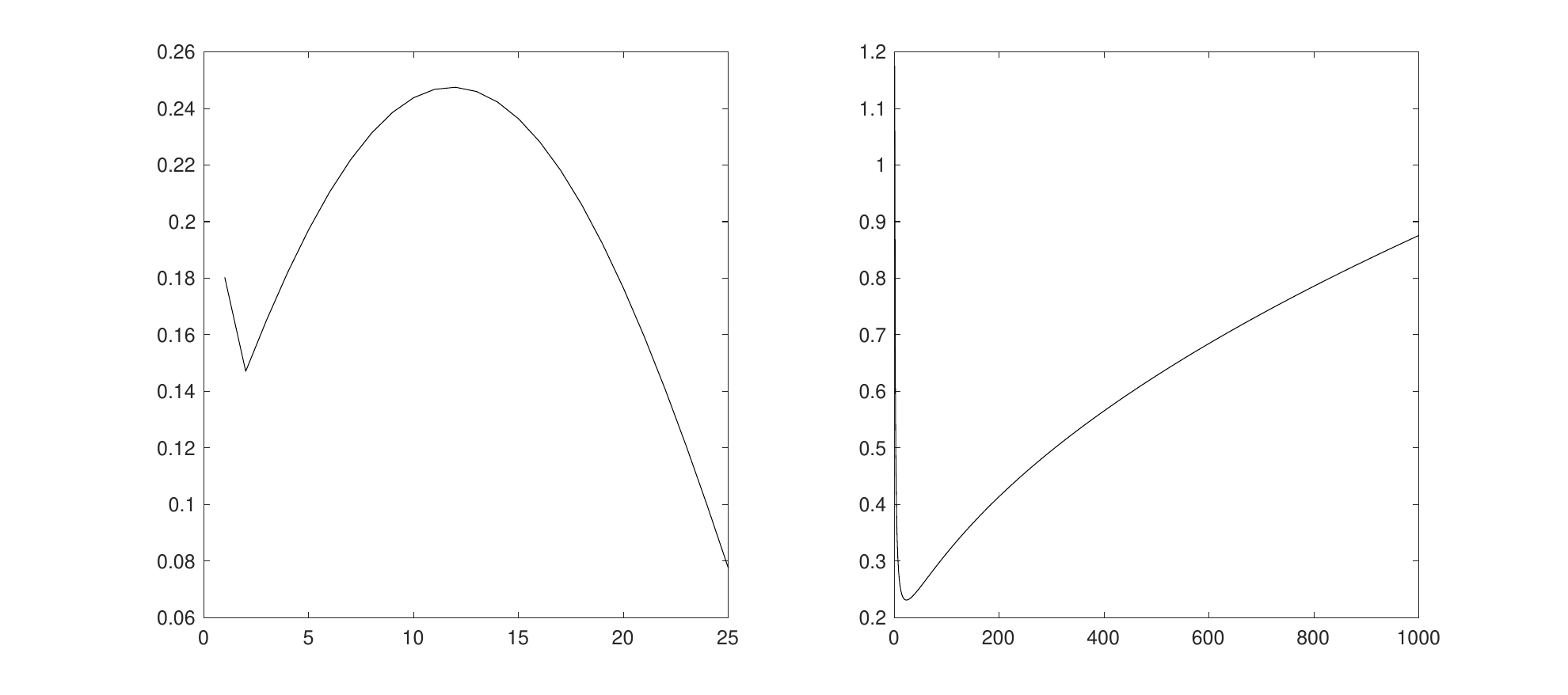}
\caption{\label{fig:optwalsh25_normvec_and_czero} Left: The level norm eigenvector $\mathbf{c}$ of the matrix $M$ in (\ref{eq:level_matrix}) for $N=24$
corresponding to $W_{24}^{\rm opt}$. Right: Plot of $1/(N c_0(N))$ for coefficient $c_0$ of level eigenvector corresponding to $W_N^{\rm opt}$ vs $N$ for $N=1,\dots, 1000$. It appears that $1/c_0$ grows   faster than $N$ but slower than $N^2$ for large $N$}
\end{figure}

When $W_N^{\rm opt}=B_{N-1,{\rm null}}$, the case in which there is no secondary branch,  $W_N^{\rm opt}\mathbf{u}=\lambda \mathbf{u}$ with unit level vector
$\mathbf{u}$.  Writing $\|W_N^{\rm opt}\mathbf{u}\|^2=\|B_{N-1,{\rm null}}(\mathbf{u})\|^2$ in  (\ref{eq:bnk_norm}) implies that the normalized left half $\mathbf{x}=\mathbf{u}_L/\alpha $ of $\mathbf{u}$ satisfies 
\begin{eqnarray*}
\|W_N^{\rm opt}\mathbf{u}\|^2 &=& \|W_{N-1}^{\rm opt}\mathbf{u}_L\|^2+\frac{1}{2}|c_N|^2+2\langle 2^{-N/2}\sum \mathbf{u}_L, 2^{(N-1)/2}c_N 2^{-N/2}\rangle
\notag \\
&=&\alpha^2\|W_{N-1}^{\rm opt}\mathbf{x}\|^2+\frac{1}{2}|c_N|^2
+2 2^{-1/2} c_N \alpha 2^{-N/2}\| \mathbf{x}\|_1\, 
\end{eqnarray*}
where  (\ref{eq:bnk_norm}) is optimized when $\alpha^2
=1-c_N^2$ and $c_N=\gamma$.

The vector $\mathbf{x}=\mathbf{u}_L/\alpha$
 is not itself an eigenvector of $W_{N-1}^{\rm opt}$  in the $K$-null case.
However, it is an approximate eigenvector in the following sense. Let $\mathbf{c}$ be the level eigenvector of $M_N$ corresponding to $\mathbf{u}$ as in (\ref{eq:level_correspondence}) (with $\mathbf{x}$ replaced by $\mathbf{u}$). Let $\mathbf{c}_L=[c_0,\dots, c_{N-1}]^T$. If $\mathbf{d}=M_{N-1}\mathbf{c}_L$ then one can check that \
\begin{equation}\label{eq:m-eigen-left}
d_k=\begin{cases} \sqrt{2}\lambda(N) c_0-c_0/(\sqrt{2}\lambda(N)), & k=0\\  \sqrt{2}\lambda(N) c_k-c_{N-k}/\sqrt{2}, & k=1,\dots, N-1 \end{cases}
\end{equation}
Suppose for the moment that the inside part $[c_1,\dots, c_{N-1}]\in\mathbb{R}^{N-1}$ of $\mathbf{c}$ is symmetric. Then we can write $[d_1,\dots,d_{N-1}]\approx (\sqrt{2}\lambda(N) -1/\sqrt{2})[c_1,\dots, c_N]$. If $\lambda(N)=(1+\sqrt{2}/2-\epsilon(N))$ where $\epsilon(N)$ is small then $(\sqrt{2}\lambda(N) -1/\sqrt{2})=(\sqrt{2}(1+\sqrt{2}/2-\epsilon_N)-1/\sqrt{2})
=1+\sqrt{2}/2-\sqrt{2}\epsilon(N)=\lambda(N)-(\sqrt{2}-1)\epsilon(N)$ since $\sqrt{2}-1/\sqrt{2}=\sqrt{2}/2$.  This would imply that $M_{N-1}\mathbf{c}_L\approx (\lambda(N)-(\sqrt{2}-1)\epsilon(N))\mathbf{c}_L+[(1/\sqrt{2}-1/(\sqrt{2}\lambda(N)))c_0,0,\dots,0]^T+(\mathbf{c}_L-\widetilde{\mathbf{c}}_L)/\sqrt{2}$  where $(\widetilde{\mathbf{c}}_L)_k=c_{N-k}$, $k=1,\dots, N$.
The norm of $\mathbf{c}_L-\widetilde{\mathbf{c}}_L$
is $0.1479$ when $N=24$ and $0.0074$ when $N=1000$.
The 
 level coefficients $\mathbf{c}(N-1,K)$ of $\mathbf{x}(N-1,K)$
are plotted for $N=24$ in Fig.~\ref{fig:nprime25_vecs}.

When a secondary branch of $B_{N-1,K}$ of width $2^K$ is introduced with small $K$, the number of coordinates with truncation length one is decreased from
$2^{N-1}$ to $2^{N-1}-2^K$ to accommodate the secondary branch. The $F_1$ analysis (\ref{eq:f1params}) suggests that $\beta\sim 2^{(K-N)/2}$ with a small proportionality constant in this case.  This is confirmed numerically in the log plots of Fig.~\ref{fig:alphabetaxyz} as functions of the parameter $K=0,\dots, 23$ with $N=24$.   

For small $K$  the expression $F_1$ also suggests that the parameter $\gamma$ should be approximately proportional to $C^2$, the number of minimal length vectors of  $B_{N-1,K}$.   Figure \ref{fig:alphabetaxyz} confirms that  the decrease in the value
 of $\gamma$ as $K$ increases is approximately proportional to the change in the number of minimal length vectors. Additionally,  the eigenvector analysis of $W_N^{\rm opt}$ above suggests that the proportionality constant should be the coefficient $c_N$ defined by (\ref{eq:m-eigenvector}). As a consequence,  the decrease in $\gamma$ exceeds  the increase in $\beta$ when $K$ is small.   
 Since $\alpha^2+\beta^2+\gamma^2=1$, this forces an increase in $\alpha$ with $K$ until $K$ is large enough that 
 the secondary branch begins to influence  the $F_2$ part of $F$ in (\ref{eq:falphabeta}).

\begin{figure*}[htp]
\centering
\begin{minipage}{.45\textwidth}
\includegraphics[width=6cm,height=4cm]{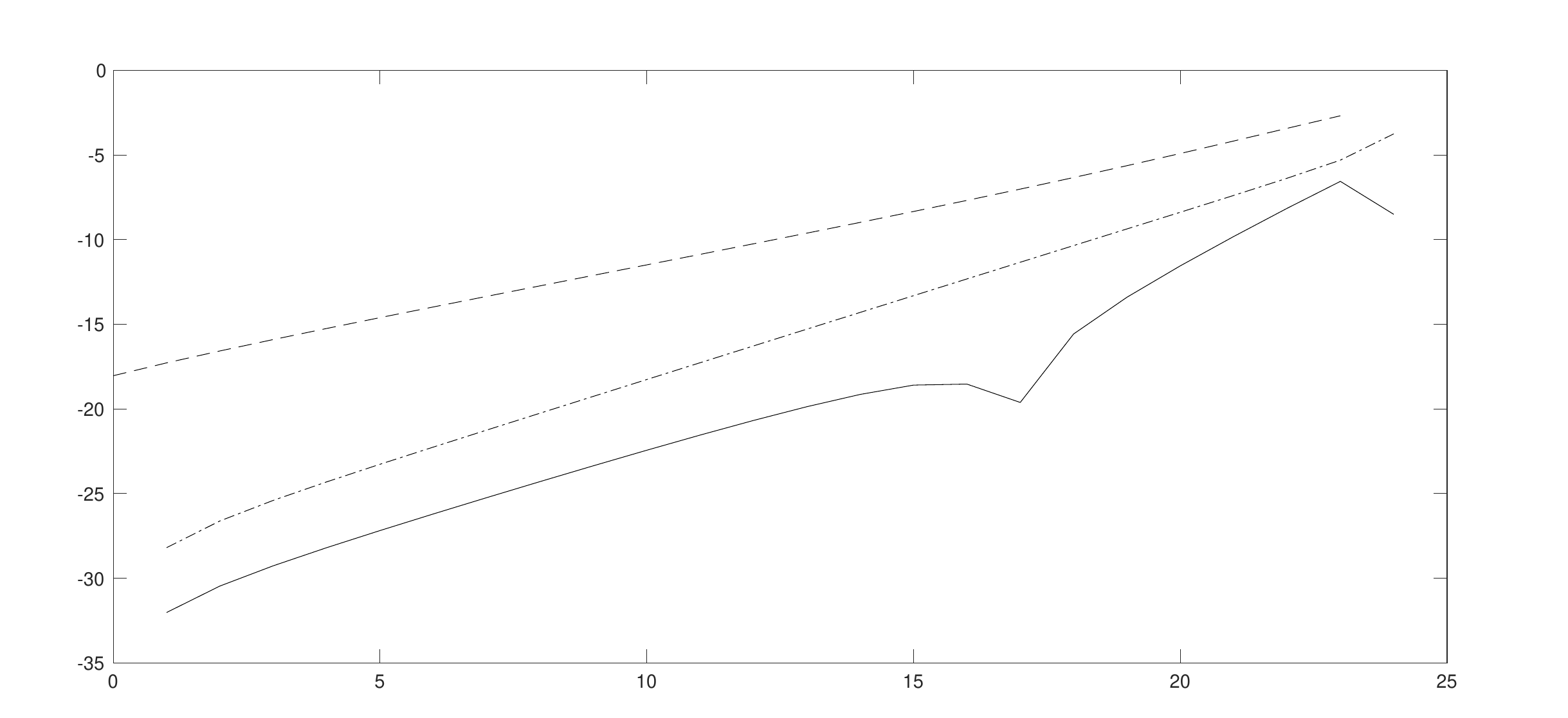}%
\caption{$\log_2$ plots of quantities $|\alpha(N,K)-\alpha(N,0)|$ (solid), $\beta(N,K) $ (dashed) and $\gamma(N,0)-\gamma(N,K)$ for $N=24$ and $K=0,\dots, 23$}
\label{fig:alphabetaxyz} 
\end{minipage}%
\hfill
\begin{minipage}{.45\textwidth}
\includegraphics[width=6cm,height=4cm]{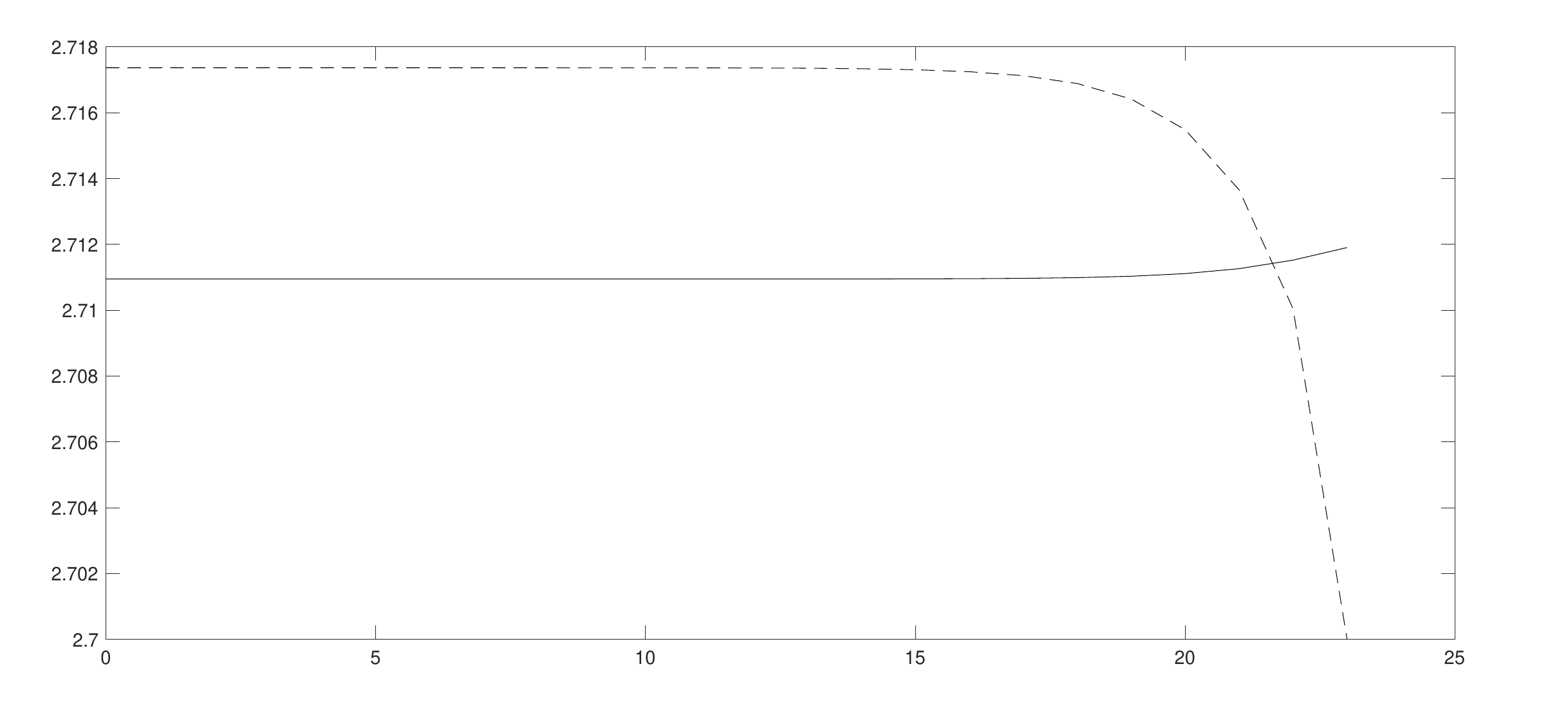}%
\caption{
Quantities $A^2(N,K)$ (solid)  and $Cx\frac{\gamma}{\alpha}+{\rm const}$ (dashed) appearing in (\ref{eq:Fcritical})
}
\label{fig:Fsumterms}
\end{minipage}%
\end{figure*}

\section{Towards a rigorous proof of optimality of $W_N^{\rm opt}$\label{sect:rigorous}}
To use Prop.~\ref{prop:Fcritical} to prove Conj.~\ref{claim:bnk_opt} requires estimates on $\alpha,\gamma,x=2^{-N/2}\|\mathbf{x}\|_1$
and $A^2=\|W_{N-1}^{\rm opt}\mathbf{x}\|_2^2$ when $\mathbf{x}$ is the left half of a vector $\mathbf{u}$ that optimizes $\|B_{N-1,K}(\mathbf{u})\|$.  The estimates need to quantify the subclaims listed below Prop.~\ref{prop:Fcritical} with sufficient precision  to confirm that the quantity $A^2+C x\frac{\gamma}{\alpha}$ in Prop.~\ref{prop:Fcritical} is decreasing in $K$
 when the quantities are considered as functions of $K$. That is, to confirm that $C x\frac{\gamma}{\alpha}$ is decreasing faster than $A^2$ is increasing, as is confirmed numerically (see Fig.~\ref{fig:Fsumterms}). As was indicated in the discussion of the subclaims, Figs.~\ref{fig:alphabetaxyz}--\ref{fig:nprime25_vecs} confirm the following numerically.
\begin{enumerate}
\item $\mathbf{x}$ evolves monotonically between being the normalized  first $2^{N-1}$ coordinates of the norm eigenvector of $W_{N}^{\rm opt}$ ($K$ null) to being the norm eigenvector of $W_{N-1}^{\rm opt}$ ($K=N-1$). This means that for each $k=0,\dots,2^{N-1}-1$, $x_k(K)$ converges monotically to $x_k(N-1)$ as $K$ increases from zero to $N-1$. Also, $\|\mathbf{x}\|_1$ is decreasing in $K$.
\item $\gamma(N,K)-\gamma(N,K+1)$ is approximately proportional to $2^{K-N}$.
\end{enumerate}
Equation \ref{eq:level_correspondence} gives an estimate of $\|\mathbf{x}\|_1$ when $\mathbf{x}$ is an eigenvector of $W_{N}^{\rm opt}$ provided one has accurate estimates of $c_0$ and $\lambda$ in (\ref{eq:m-eigenvector}).  Besides estimates of the level coefficients $c_0$ and $c_N$, one also needs precise quantification of multiplicative  factors in the estimates $\beta(N,K)\sim 2^{(K-N)/2}$ and $\gamma(N,K+1)-\gamma(N,K)\sim 2^{(K-N)}$, and  of the error between $\mathbf{x}(N,K)$ and $\mathbf{x}(N,K+1)$, the left halves of successive singular vectors of  $B_{N-1,K}$.

\begin{figure}[ht]
\begin{center}
\noindent
 \includegraphics[width=12cm,height=4cm]{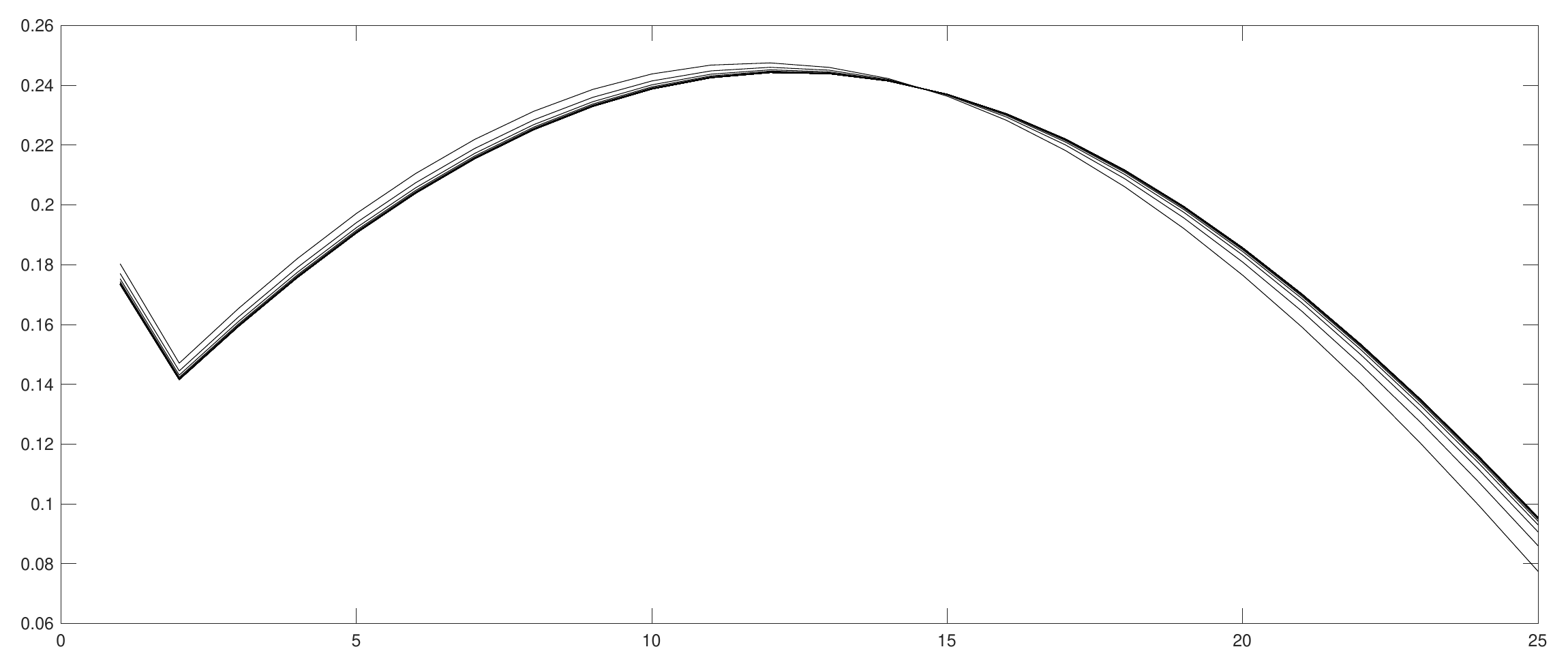}
\caption{\label{fig:nprime25_vecs} Vectors $\mathbf{c}(N-1,K)$ for $N=25$ and $K=0,\dots 24$ where $\mathbf{c}$ are the level coefficients of the vector $\mathbf{x}(N-1,K)$ of coefficients of the  primary branch of the norm singular vector of  $B_{N-1,K}$. The solid-looking part comes from the \emph{small} values of $K$ where the  coefficients $c_k(N-1,K)$ evolve slowly, most of the change occurring in the values $K=20,\dots, 24$. For the matrix $M$ corresponding to $N-1=24$ the norm of $M\mathbf{c}(24,0)$ is $1.6464\dots$ while
that of  $M\mathbf{c}(24,24)$  (the norm eigenvector of $W_{24}^{\rm opt}$) is $1.6468\dots$}
\end{center}
\end{figure}

Recall that as $K$ ranges from \emph{null} to $N-1$, the $\mathbf{x}$-part of an optimizer for $B_{N-1,K}$ ranges from being the first $2^{N-1}$ coordinates of 
an eigenvector of $W_N^{\rm opt}$ 
to being an eigenvector of $W_{N-1}^{\rm opt}$. 
In the latter case one has $\|\mathbf{x}(N-1,N-1)\|_1=2^{N/2} \lambda(N-1) c_0(N-1)$ whereas in the former case, by (\ref{eq:m-eigenvector}), one has  $\|\mathbf{x}(N-1,{\rm null})\|_1=2^{N/2} \lambda(N) c_1(N)/\sqrt{1-c_N^2(N)}$ 
where the denominator reflects that $\mathbf{x}(N-1,{\rm null})$ is the $\ell^2$-renormalization of the left half of the norm eigenvector of $W_N^{\rm opt}$ and $\mathbf{c}$ is its 
level vector. 
Values $\|\mathbf{x}(N-1,K)\|_1$ are difficult to quantify explicitly when $K<N-1$, but can be done  in terms of  the singular value decomposition of $B_{N-1,K}$.
Coordinates of $\mathbf{x}(N-1,K)$ can be shown to evolve monotonically between those of $\mathbf{x}(N-1,{\rm null})$ and $\mathbf{x}(N-1,N-1)$.

The eigenvector equations (\ref{eq:m-eigenvector}) can be re-expressed as
\begin{eqnarray} \label{eq:m-eigenvector2}
c_{1}&=& c_0- \mu  c_{N} ;\quad \mu=1/(\sqrt{2} \lambda(N)) \notag\\
c_{k+1}&=&\sqrt{2} c_k- \mu  c_{N-k} ;\quad k=1,\dots, N-1 \, .
\end{eqnarray}
The first few iterates give $c_N=\mu c_0$; $c_1=(1-\mu^2)c_0$; $c_{N-1}=2^{-1/2}\mu (2-\mu^2)c_0$, $c_2=[\sqrt{2} (1-\mu^2)- 2^{-1/2}\mu^2 (2-\mu^2)]c_0$ and so on.
Iterating leads to expressions of $c_k$ ($k=1,\dots, N$)   as polynomials in $\mu$, 
\begin{equation} \label{eq:m-eigenvector-coefficient-expansion}
c_k=\sum_{\ell=0}^k c_{k,\ell} \mu^{2\ell};\quad
c_{N-k}=\sum_{\ell=0}^k d_{k,\ell} \mu^{2\ell+1}
\end{equation}
whose coefficients $c_{k,\ell}$ and $d_{k,\ell}$ satisfy
\begin{eqnarray}  \label{eq:m-eigenvector-coefficient-expansion-recursion}
c_{k+1,\ell}&=&\sqrt{2} c_{k,\ell}-\, d_{k,\ell-1}, (k>0) \notag \\
 d_{k+1,\ell}&=& 2^{-1/2}( d_{k,\ell}+c_{k+1,\ell}), (k<N-1) \, .
 \end{eqnarray}
In theory these lead to formulas expressing each $c_k$ as $c_0$ times a polynomial of degree $2k$ in $\mu$ whose coefficients are sums of powers of $\sqrt{2}$.
Closed forms would then allow for explicit calculation of $c_0$ in terms of $\mu$ using $\|\mathbf{c}\|_2=1$ and, in turn, exact expressions for $\alpha,\gamma$ in the case $K=0$ (and $K=N-1$).

The eigenvector analysis of $W_N^{\rm opt}$ is possible because $W_N^{\rm opt}$ and its level matrix $M_N$ are symmetric.
The matrix $B_{N-1,K}$ is not symmetric. Nevertheless, its singular vectors are constant on indices corresponding to columns of
equal truncation length, allowing for such vectors to be expressed in terms of an analogous level matrix (of size $(N+K+2)\times (N+K+2)$ to decouple the two mutually orthogonal branches). This level matrix also is not symmetric and  the parameters to optimize $B_{N-1,K}$ in (\ref{eq:bnk_norm2}) cannot be determined exactly for general
$K$ in the same way as above for $K=0,N-1$.
Nevertheless, the eigen-analysis of $W_N^{\rm opt}$ can be adapted, with some small quantifiable errors,
 to a change in the number of minimal-length columns from
$2^{N-1}$ to $2^{N-1}-2^K$. 
This, together with exact formulas for level coefficients of norm singular vectors, should  provide sufficient precision in estimating the optimizing parameters
to conclude that the increase of $A^2$ in  $K$ in (\ref{eq:Fcritical}) is more than compensated by a decrease in $Cx\frac{\gamma}{\alpha}$
as a function of $K$, resulting in a net decrease in $F$ as a function of $K$ evaluated at parameters that optimize $B_{N-1,K}$.

\section{Discussion: Non-dyadic truncations\label{sect:discussion}}

Figure \ref{fig:biopt33_trim} shows a matrix that we denote $B_{3,3}^{\rm trim}$ that is obtained from $B_{3,3}$ (see (\ref{eq:bnk_norm}))  by trimming any columns whose
bottom nonzero entries are equal to $-1$, setting those entries to zero so the new bottom entry is equal to one.  The matrix on the left in Fig.~\ref{fig:biopt33_trim} has norm equal to $1.31\dots$ which is also equal to the norm of $W_3^{\rm opt}$. The norm
of $B_{3,3}^{\rm trim}$ is slightly larger,  $1.361\dots$ but smaller than the norm of the  standard truncation of the same dimensions, namely the $16\times 16$ matrix $W_4^{\rm opt}$ which is $1.366\dots$.  It turns out, however, that $\|W_{N}^{\rm opt}\|<\|B_{N-1,N-1}^{\rm trim}\|$, where both matrices have size $2^N\times 2^N$, as soon
as $N\geq 7$, as can be verified numerically. For example, in the $128\times 128$ case, $\|W_7^{\rm opt}\|=1.4739\dots$ whereas $\|B_{6,6}^{\rm trim}\|=1.4746\dots$.  This is related to their \emph{total correlations} 
where the total correlation of a matrix $A\in\mathbb{R}^{m\times n}$ with columns $\mathbf{a}_j$ is $T_A=\sum_i\sum_j\langle \mathbf{a}_i,\, \mathbf{a}_j\rangle
=\mathbf{1}^T A^T A\mathbf{1}$ where $\mathbf{1}\in\mathbb{R}^n$ is the vector whose entries are one in each coordinate.
The total correlation  of $W_7^{\rm opt}$ is $T_{W_7^{\rm opt}}=191.50\dots$, larger than 
 $T_{B_{6,6}}=191.00\dots$, but $T_{W_7^{\rm opt}}<T_{B_{6,6}^{\rm trim}}=201.67\dots$.  While the trimming diminishes
correlations among columns on the right, that loss is more than compensated by new correlation added between pairs of columns on the left and right sides.  A gain in total correlation of course does not imply a gain in norm and, in fact,
 the norm gain is only marginal. One can check for example that $\|W_8^{\rm opt}\|=1.498\dots >\|B_{6,6}^{\rm trim}\|$. We have verified numerically that 
$\|W_{N+1}^{\rm opt}\|>\|B_{N-1,N-1}^{\rm trim}\|$ for computable $N$.  
For larger $N$, one can obtain non-dyadic truncations with  larger norm by introducing further \emph{trimmed bifurcations}.
Again, the gain is marginal. Numerical
investigation suggests that if $W_{N,\Phi}$ is a $2^{N}\times 2^{N}$ (trimmed) TWH matrix that has $P$ nodes then $\|W_{N+P+1}^{\rm opt}\|>\|W_{N,\Phi}\|$.  
If one has a \emph{dilation method} to  improve norm successively by doubling  dimension and replacing blocks of columns {having } a single bifurcation by a corresponding set of columns of \emph{standard type} forming a single branch, then one should be able to conclude a strong form of  Conj.~\ref{conj:uniform_bound_general}, namely that 
$\|W_\Phi\|\leq 1+\frac{\sqrt{2}}{2}$ for any truncation $\Phi$.  
The weaker form of  Conj.~\ref{conj:uniform_bound_general} as stated is based on a less delicate approach. {Measure-theoretic arguments would also be needed } to allow for truncation maps $\Phi$ in Conj.~\ref{conj:uniform_bound_general} that do not necessarily belong to $\mathcal{D}_N$ for some $N$.

\begin{figure}[ht]
\begin{center}
\noindent
 \includegraphics[width=12cm,height=4cm]{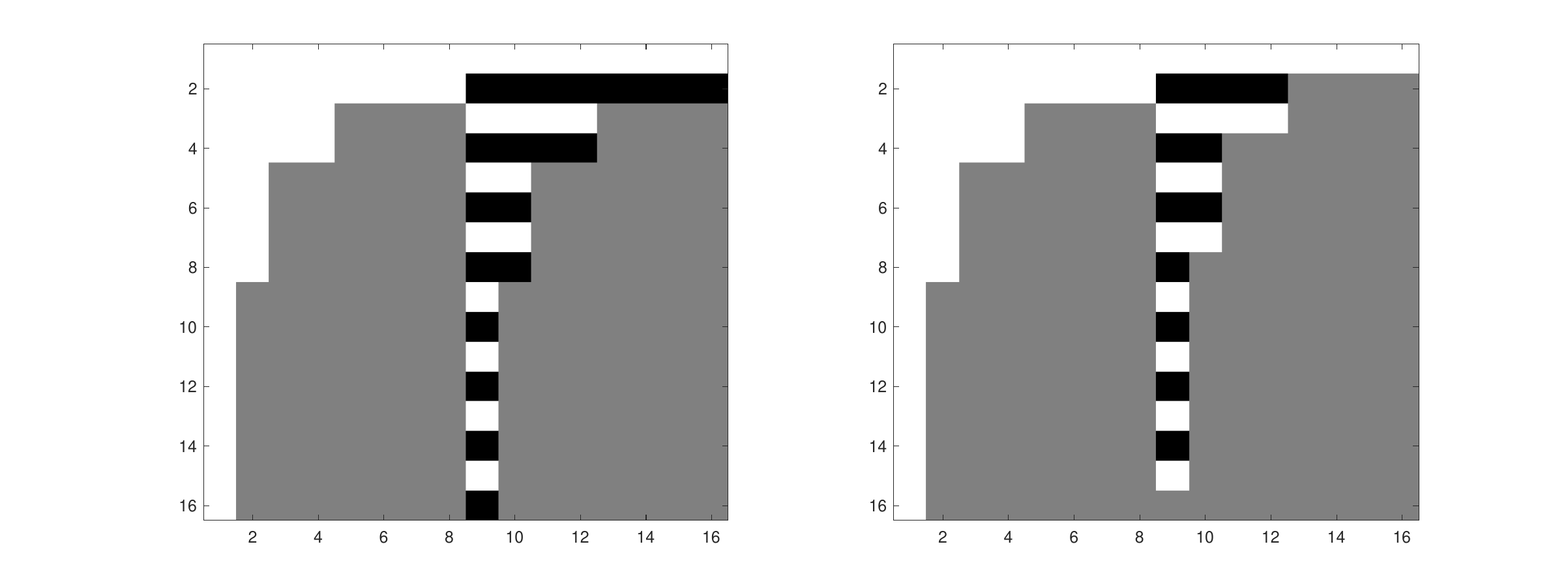}
\caption{\label{fig:biopt33_trim} Left: $16\times 16$ matrix $B_{3,3}$. Right: Matrix $B_{3,3}^{\rm trim}$ obtained from $B_{3,3}$ by trimming the bottom
entries of $B_{3,3}$ that are equal to $-1$}
\end{center}
\end{figure}

\appendix

\section{Appendix: Proof of  Prop.~\ref{prop:Fcritical}\label{appendix:proof_Fcritical}}

 \begin{proof}[Proof of Prop.~\ref{prop:Fcritical}]
 For 
 \[F(\alpha,\beta;x,y)=\alpha^2 A^2+\beta^2 B^2+\gamma^2 C^2 +2\gamma C(\alpha x+\beta y)
\]
 we have
 \[\frac{\partial F}{\partial \alpha}=2\alpha (A^2-C^2)+2\gamma C x-2C\frac{\alpha}{\gamma}(\alpha x+\beta y)\,\,{\rm and}
 \]
  \[\frac{\partial F}{\partial \beta}=2\beta (B^2-C^2)+2\gamma C y-2C\frac{\beta}{\gamma}(\alpha x+\beta y)\, .
 \]
Vanishing of both terms yields
  \[0=2\alpha (A^2-C^2)+2\gamma C x-2C\frac{\alpha}{\gamma}(\alpha x+\beta y)
 \]
\[ 2\alpha A^2+2\gamma Cx=2\alpha C^2+2C\frac{\alpha}{\gamma}(\alpha x+\beta y)
\]
 \begin{equation}\label{eq:dfdalpha=0} A^2+\frac{\gamma}{\alpha} Cx= C^2+\frac{C}{\gamma}(\alpha x+\beta y)\, .
 \end{equation}
 Similarly
 \begin{equation}\label{eq:dfdbeta=0} B^2+\frac{\gamma}{\beta} Cy= C^2+\frac{C}{\gamma}(\alpha x+\beta y)\, .
 \end{equation}
 These expressions allow us to eliminate $B^2$ and $y$.   Specifically, by (\ref{eq:dfdalpha=0}) we get
 \begin{equation}\label{eq:cbetay} \frac{C}{\gamma}(\beta y)= A^2-C^2+ Cx\left(\frac{\gamma}{\alpha} -\frac{\alpha}{\gamma}\right)
 \end{equation}
 and by (\ref{eq:dfdbeta=0})
 \begin{equation}\label{eq:bsquared} B^2= C^2+C\left(\frac{\alpha}{\gamma} x+\left(\frac{\beta}{\gamma}-\frac{\gamma}{\beta}\right) y\right)\, .
 \end{equation}
Using  (\ref{eq:bsquared}) to substitute for $B^2$  into the expression for $F$ yields
 \begin{multline*}F(\alpha,\beta;x,y)=\alpha^2 A^2+\beta^2 \left(C^2+C\left(\frac{\alpha}{\gamma} x+\left(\frac{\beta}{\gamma}-\frac{\gamma}{\beta}\right) y\right)\right)
 +\gamma^2 C^2 +2\gamma C(\alpha x+\beta y)\\
 =\alpha^2 A^2+\beta^2 \left(C^2+C\left(\frac{\alpha}{\gamma} x+\left(\frac{\beta}{\gamma}-\frac{\gamma}{\beta}\right) y\right)\right)
 +\gamma^2 C^2 +2\gamma^2 C\left(\frac{\alpha}{\gamma} x+\frac{\beta}{\gamma} y\right)\\
  =\alpha^2 A^2+(\beta^2+\gamma^2)C^2+( \beta^2+2\gamma^2)C\frac{\alpha}{\gamma} x+
  \beta^2 C\left(\frac{\beta}{\gamma}\left(1-\frac{\gamma^2}{\beta^2}\right) y\right)
+2\gamma^2 C\left(\frac{\beta}{\gamma} y\right)\\
\end{multline*}
Applying (\ref{eq:cbetay}) then gives
 \begin{multline*}F(\alpha,\beta;x,y)
  =\alpha^2 A^2+(\beta^2+\gamma^2)C^2+( \beta^2+2\gamma^2)C\frac{\alpha}{\gamma} x+C\frac{\beta}{\gamma}y(\beta^2+\gamma^2)\\
  =\alpha^2 A^2+(\beta^2+\gamma^2)C^2+( \beta^2+2\gamma^2)C\frac{\alpha}{\gamma} x
  +\left(A^2-C^2+ Cx\left(\frac{\gamma}{\alpha} -\frac{\alpha}{\gamma}\right)\right)(\beta^2+\gamma^2)\\
    =(\alpha^2+\beta^2+\gamma^2) A^2+( \beta^2+2\gamma^2)C\frac{\alpha}{\gamma} x
  +\left( Cx\left(\frac{\gamma}{\alpha} -\frac{\alpha}{\gamma}\right)\right)(\beta^2+\gamma^2)\\
      =A^2+ \beta^2 C\frac{\alpha}{\gamma} x+2\gamma^2 C\frac{\alpha}{\gamma} x
  +\left( Cx\left(\frac{\gamma}{\alpha} -\frac{\alpha}{\gamma}\right)\right)(\beta^2)
    +\left( Cx\left(\frac{\gamma}{\alpha} -\frac{\alpha}{\gamma}\right)\right)(\gamma^2)\\
       =A^2+2\gamma^2 C\frac{\alpha}{\gamma} x
  +\left( Cx\left(\frac{\gamma}{\alpha} \right)\right)(\beta^2)
    +\left( Cx\left(\frac{\gamma}{\alpha} -\frac{\alpha}{\gamma}\right)\right)(\gamma^2)\\
           =A^2+\gamma^2 C\frac{\alpha}{\gamma} x
  +Cx \frac{\gamma}{\alpha}(\beta^2+\gamma^2)\\
             =A^2+C\alpha^2\frac{\gamma}{\alpha} x
  +Cx \frac{\gamma}{\alpha}(\beta^2+\gamma^2) =A^2+C \frac{\gamma}{\alpha} x\, .
\end{multline*}
The other identity in (\ref{eq:Fcritical}) follows by combining (\ref{eq:dfdalpha=0}) and (\ref{eq:dfdbeta=0}).
 \end{proof}
 
%


\end{document}